\documentclass[11 pt]{amsart}
\usepackage[latin1]{inputenc}
\usepackage{amsfonts,amssymb,amsmath,amsthm}
\usepackage{hyperref}
\usepackage{geometry}
\usepackage{graphicx, psfrag}
\usepackage{color}
\usepackage{setspace}

\newtheorem{theorem}{Theorem}
\newtheorem*{acknow}{Acknowledgement}

\newtheorem{corollary}[theorem]{Corollary}

\newtheorem{lemma}[theorem]{Lemma}

\newcommand{\N}{\mathbb{N}}

\input{tcilatex}

\begin{document}

\title[Critical BPRE with immigration:
survival of a single family]{Critical branching processes in random environment with immigration:
survival of a single family}
\author{C. Smadi}
\address{Charline Smadi, Department of Discrete Mathematics, Steklov Mathematical Institute of Russian Academy of Sciences, 8 Gubkin
Street, 117 966 Moscow GSP-1, Russia}
\email{charline.smadi@irstea.fr}
\author{V. A. Vatutin}
\address{Vladimir A. Vatutin, Department of Discrete Mathematics, Steklov Mathematical Institute of Russian Academy of Sciences, 8 Gubkin
Street, 117 966 Moscow GSP-1, Russia}
\email{vatutin@mi.ras.ru}

\maketitle

\begin{abstract}
We consider a critical branching process in an i.i.d. random environment, in which 
one immigrant arrives at each generation. We are interested in the event $\mathcal{A}_i(n)$ that 
all individuals alive at time $n$ are offspring of the immigrant which joined the population at time $i$.
We study the asymptotic probability of this event when $n$ is large and $i$ follows 
different asymptotics which may be related to $n$ ($i$ fixed, close to $n$, or going to infinity but far from $n$).
In order to do so, we establish some
conditional limit theorems for random walks, which are of independent interest.\\



\end{abstract}

\section{Introduction and main result}

We consider a branching process with immigration evolving in a random
environment. Individuals in such a process reproduce independently of each other according to
random offspring distributions which vary from one generation to the other.
In addition, an immigrant enters the population at each generation. To give
a formal definition let $\Delta $ be the space of all probability measures
on $\mathbb{N}_{0}:=\{0,1,2,\ldots \}.$ Equipped with the
metric of total variation $\Delta $ becomes a Polish space. Let $F$ be a
random variable taking values in $\Delta $, and let $F_{n},n\in \mathbb{N}:=%
\mathbb{N}_{0}\backslash \left\{ 0\right\} $ be a sequence of independent
copies of $F$. The infinite sequence $\mathcal{E}=\left\{ F_{n},n\in \mathbb{%
N}\right\} $ is called a random environment.

A sequence of $\mathbb{N}_{0}$-valued random variables $\mathbf{Y}=\left\{
Y_{n},\ n\in \mathbb{N}_{0}\right\} $ specified on\ the respective
probability space $(\Omega ,\mathcal{F},\mathbf{P})$ is called a branching
process with one immigrant in random environment (BPIRE), if $Y_{0}$ is
independent of $\mathcal{E}$ and, given $\mathcal{E}$ the process $\mathbf{Y}
$ is a Markov chain with 
\begin{equation}
\mathcal{L}\left( Y_{n}|Y_{n-1}=y_{n-1},\mathcal{E}=(f_{1},f_{2},...)\right)
=\mathcal{L}(\xi _{n1}+\ldots +\xi _{ny_{n-1}}+1)  \label{BasicDefBPimmigr}
\end{equation}%
for every $n\in \mathbb{N}$, $y_{n-1}\in \mathbb{N}_{0}$ and $%
f_{1},f_{2},...\in \Delta $, where $\xi _{n1},\xi _{n2},\ldots $ are i.i.d.
random variables with distribution $f_{n}.$ In the language of branching
processes $Y_{n-1}$ is the $(n-1)$th generation size of the population and $%
f_{n}$ is the offspring distribution of an individual at generation $n-1$.
For the sake of simplicity, we consider that if $Y_{n-1}=y_{n-1}>0$ is the
population size of the ($n-1)$th generation of $\mathbf{Y}$ then first $\xi
_{n1}+\ldots +\xi _{ny_{n-1}}$ individuals of the $n$th generation are born
and afterwards one immigrant enters the population.\newline

We will call an $(i,n)$-clan the set of individuals alive at generation $n$ and being
children of the immigrant which entered the population at generation $i$. We
say that only the $(i,n)$-clan survives in $\mathbf{Y}$ at moment $n$ if $%
Y_{n}^{-}:=\xi _{n1}+\ldots +\xi _{ny_{n-1}}>0$ and all $Y_{n}^{-}$
particles belong to the $(i,n)$-clan. Let $\mathcal{A}_{i}(n)$ be the event
that only the $(i,n)$-clan survives in $\mathbf{Y}$ at moment $n$. The aim
of this paper is to study the asymptotic behavior of the probability $%
\mathbf{P}\left( \mathcal{A}_{i}(n)\right) $ as $n\rightarrow \infty $ and $i
$ varies with $n$ in an appropriate way. It is a natural question when we
are concerned with the diversity of a population. If we assume, for
instance, that each immigrant has a new type or belongs to a new species, the
realisation of the event $\mathcal{A}_{i}(n)$ means that the entire
population is of the same type or species, and its probability quantifies
the distribution of the time of the most recent common ancestor, in the case where there is only one founder.
For instance it has been shown recently that the current invasion of France by the yellow-legged hornet is due 
to a single female which gave birth after its arrival, probably in 
horti-cultural pots carried on cargo boats from China \cite{V06,A15}.
A natural further question could be to investigate the law of the random environment
close to the time of arrival of this founder. 
Indeed, due to the stochasticity of the environment, the fate of a mutant 
strongly depends on the time of its arrival (see for instance \cite{HI96} for 
biological implications of this fact).
The law of
the immigration has been chosen simple for technical reasons, but this works
constitutes a first step in the study of more general immigration laws.%
\newline

Populations of many species are maintained by recurrent events
of extinction of local populations and subsequent invasions from other
populations (e.g., \cite{H92}), and BPRIE's appear as a natural population model.
They have been studied in several papers (see \cite{K73,KKS75,AFa2013,AFa2015,AFa2018,AFa2019} in particular) in the context of random walks in random 
(in space) environment. Indeed there is an equivalence between the law of the hitting time of an integer $n$ by a random 
walk in random environment and the total progeny of a BPRIE up to generation $n$.
Haccou and coauthors \cite{HI96,HV03} studied the establishment probability of a population modeled by a BPRIE with the question of invasions in mind.
In particular, they proved that sequential invasions have a higher probability than simultaneous invasions.
In \cite{DLVZ19,LVZ19}, the authors have studied the tail distribution of the life-periods of BPRIE's in the critical and subcritical cases, respectively, and 
with an immigration law more general than in our case. 
Up to our knowledge, the properties of the events $\mathcal{A}_i(n)$, quantifying in some sense the low diversity of the population
at large times, has not been investigated until now.\\

We consider, along with the process $\mathbf{Y}$, a standard branching
process $\mathbf{Z}=\left\{ Z_n,\ n\in \mathbb{N}_{0}\right\} $ in the
random environment $\mathcal{E} $ which, given $\mathcal{E}$ is a Markov
chain with $Z_0=1$ and
\begin{equation}
\mathcal{L}\left( Z_{n}|Z_{n-1}=z_{n-1},\mathcal{E}=(f_{1},f_{2},...)\right)
=\mathcal{L}(\xi _{n1}+\ldots +\xi _{nz_{n-1}})  \label{BPordinary}
\end{equation}
for $n\in \mathbb{N}$, $z_{n-1}\in \mathbb{N}_{0}$ and $f_{1},f_{2},...\in
\Delta $.

To formulate our results we introduce the so-called associated random walk $%
\mathbf{S}=\left( S_{n},n\in \mathbb{N}_{0}\right) $. This random walk has
increments $X_{n}=S_{n}-S_{n-1}$, $n\geq 1$, defined as 
\begin{equation*}
X_{n}=\log m\left( F_{n}\right)
\end{equation*}%
which are i.i.d. copies of the logarithmic mean offspring number $X:=\log $ $%
m(F)$ with%
\begin{equation*}
m(F):=\sum_{j=1}^{\infty }jF\left( \left\{ j\right\} \right) .
\end{equation*}%
%
%
%

With each measure $F$ we associate the respective probability generating
function 
\begin{equation*}
F(s)=\sum_{j=0}^{\infty }F\left( \left\{ j\right\} \right) s^{j}.
\end{equation*}%
We assume that the probability generating functions meet the following
restrictions.\newline

\noindent \textbf{Hypothesis A1}. The probability generating function $F(s)$
is geometric with probability 1, that is%
\begin{equation}
F(s)=\frac{q}{1-ps}=\frac{1}{1+m(F)(1-s)}  \label{Frac_generating}
\end{equation}%
with random $p,q\in (0,1)$ satisfying $p+q=1$ and 
\begin{equation*}
m(F)=\frac{p}{q}=e^{\log (p/q)}=e^{X}.
\end{equation*}

\noindent \textbf{Hypothesis A2}. The branching process in random
environment is critical and satisfies the following moment conditions: 
\begin{equation*}
\mathbf{E}\left[ X\right] =0,\quad \mathbf{E}\left[ X^{2}\right] \in \left(
0,\infty \right) \quad \text{and}\quad \mathbf{E}\left[ e^{X}+e^{-X}\right]
<\infty .
\end{equation*}

\noindent \textbf{Hypothesis A3}. The distribution of $X$ is absolutely
continuous.\newline

Recall that $\mathcal{A}_{i}(n)$ is the event that only the $(i,n)$-clan
survives in $\mathbf{Y}$ at moment $n$. The main result of this paper
provides the asymptotic behavior of the probability $\mathbf{P}\left( 
\mathcal{A}_{i}(n)\right) $ as $n\rightarrow \infty $ and $i$ varies with $n$
in an appropriate way.

\begin{theorem}
\label{T_total}If Hypotheses A1-A2 are valid then

1) for any fixed $i$%
\begin{equation}
\lim_{n\rightarrow \infty }n^{3/2}\mathbf{P}\left( \mathcal{A}_{i}(n)\right)
=w_{i}\in \left( 0,\infty \right) ;  \label{AsumBegin}
\end{equation}

2) for any fixed $N$%
\begin{equation*}
\lim_{n\rightarrow \infty }n^{1/2}\mathbf{P}\left( \mathcal{A}%
_{n-N}(n)\right) =r_{N}\in \left( 0,\infty \right) ;
\end{equation*}

3) if, in addition, Hypothesis A3 is valid then 
\begin{equation}
\lim_{\min \left( i,n-i\right) \rightarrow \infty }i^{1/2}\left( n-i\right)
^{3/2}\mathbf{P}\left( \mathcal{A}_{i}(n)\right) =K\in \left( 0,\infty
\right) .  \label{AsymInterm}
\end{equation}
\end{theorem}

These results may seem counter-intuitive at a first glance. 
For instance, if we take $i=n/2$ in point 3), we obtain
$$ \lim_{n \rightarrow \infty }n^2 \mathbf{P} \left( \mathcal{A}_{n/2}(n)\right) =4K. $$
And thus the event $\mathcal{A}_i(n)$ for a fixed $i$ is more likely to happen than the event $\mathcal{A}_{n/2}(n)$, whereas the 
clan $(i,n)$ has to survive longer than the clan $(n/2,n)$. But we can get more intuition on the result by thinking  in terms of associated random 
walk $\mathbf{S}$.

As previously observed under different assumptions on the random environment
(see, for instance, \cite{vatutin2013evolution} for a comprehensive review
on the critical and subcritical cases (before 2013) or the recent monograph
\cite{GV2017}) the survival of a branching process in random environment until a given time $n$ is
essentially determined by its survival until the moment when the associated
random walk $\mathbf{S}$ attains its infimum on the interval $[0,n]$. The idea is that if we divide
the trajectory of the process on the interval $[0,n]$ into two parts, one
before the running infimum of the random environment $\mathbf{S}$, and one
after this running infimum, the process will live in a favorable environment
after the running infimum of the random environment, and will thus survive
with a non-negligible probability until time $n$, provided it survived until
the time of the running infimum.
But for the survival of the population to be likely until the time of the infimum of the random environment on $[0,n]$,
this infimum should not take too small values. To precise this intuition, let us recall the two main results of 
\cite{4h} which concerns critical branching processes in random environment without immigration. First of all, Theorem 1.1 in \cite{4h} states that there exists 
a positive and finite constant $\theta$ such that: 
$$ \mathbf{P}(Z_n>0) \sim \theta \mathbf{P}(\min(S_0,...,S_n)\geq 0), \quad \text{as} \quad n \to \infty. $$
Theorem 1.4 in \cite{4h} states that the law of $(\tau(n),\min(S_0,...,S_n))$ conditionally on the event $\{Z_n>0\}$ converges weakly to some probability 
measure on $\mathbb{N}_0 \times \mathbb{R}_0^-$, where $\tau(n)$ is the time when the random walk $\mathbf{S}$ reaches its minimum on $[0,n]$.
In particular, it implies that for any $\varepsilon>0$ there exists $x(\varepsilon)>0$ such that for any $x \geq x(\varepsilon)$ and $n$ large enough,
$$ \mathbf{P}(Z_n>0, \min(S_0,...,S_n)\leq -x) \leq \varepsilon \mathbf{P}(Z_n> 0). $$

As a consequence, for an immigrant arriving at generation $i$ to be the only ancestor of the population at time $n$ we have to combine 
(at least) two elements: first the random environment has to be bad enough before time $i$ for the other families alive before time $i$
to get extinct (broadly speaking $\min(S_0,...,S_i)\leq 0$, and even $\tau(i)$ close ot $i$); second the environment has to be good enough 
after time $i$ for the $(i,n)$-clan to stay alive (broadly speaking $\min (S_{i+k}-S_i, 0 \leq k \leq n-i)\geq 0$). Of course the 
analysis of the probability of the event $\mathcal{A}_i(n)$ is more involved, as we also have to take into account the fact that the 
subsequent clans do not survive, but, as we will now show, this analysis give the good order of magnitude for the part before time $i$.
The probability that the minimum of the random walk on the time interval $[0,n]$ is reached at time $i$ is:
\begin{align*} \mathbf{P}(\tau(n)=i)&= \mathbf{P}(\tau(i)=i)\mathbf{P}(\min(S_0,...,S_{n-i})\geq 0)\\
&=\mathbf{P}(\max(S_1,...,S_{i})< 0)\mathbf{P}(\min(S_0,...,S_{n-i})\geq 0) \\
&\sim \frac{C}{i^{1/2}(n-i)^{1/2}} \quad \text{as} \quad (i,n-i) \to \infty, \end{align*}
where we have used duality principle for random walks (see Theorem 4.1 in \cite{GV2017}) and Equation \eqref{Asym0} below.
Such an analysis is enough to understand the leading order for point 2) of the theorem, but a more thorough analysis is required to 
understand the term of order $(n-i)^{-3/2}$, describing the fact that only the $(i,n)$-clan survives. This will be the aim of the subsequent proofs.\\

The rest of the paper is organised as follows. In Section \ref{sec_aux_res} we collect some auxiliary results dealing with explicite expressions for 
the probability of the event $\mathcal{A}_i(n)$. Section \ref{sec_nN} is dedicated to the proof of point 2) of Theorem \ref{T_total}.
The proof of point 1) of Theorem \ref{T_total} is provided in Section \ref{sec_fixed_i}.
Finally the proof of Theorem \ref{T_total} is completed in Section \ref{sec_ini_to_infty}.\\

In the sequel we will denote by $C,C_{1}, C_{2},...$ constants which may vary from line to line
and by $K_{1},K_{2},...$ some fixed
constants.

\section{Auxiliary results} \label{sec_aux_res}

Given the environment $\mathcal{E}=\left\{ F_{n},n\in \mathbb{N}\right\} $,
we construct the i.i.d. sequence of generating functions 
\begin{equation*}
F_{n}(s):=\sum_{j=0}^{\infty }F_{n}\left( \left\{ j\right\} \right)
s^{j},\quad s\in \lbrack 0,1],
\end{equation*}%
and use below the convolutions of $F_{1},...,F_{n}$ specified for $0\leq
i\leq n-1$ by the equalities $F_{n,n}(s):=s$, 
\begin{equation*}
\left\{ 
\begin{array}{l}
F_{i,n}(s) :=F_{i+1}(F_{i+2}(\ldots F_{n}(s)\ldots )),\quad \\ 
F_{n,i}(s) :=F_{n}(F_{n-1}(\ldots F_{i+1}(s)\ldots )).%
\end{array}
\right.
\end{equation*}

Then we can express the probability of the event $\mathcal{A}_i(n)$ conditionally on $\mathbf{S}$ as follows:
\begin{equation}  \label{expr_PAin}
\mathbf{P}\left( \mathcal{A}_{i}(n)| \mathbf{S}\right) =\mathbf{E}\left[
(1-F_{i,n}(0))\prod_{k\neq i}^{n-1}F_{k,n}(0) \Bigg| \mathbf{S} \right] .
\end{equation}

For the sake of readability, let us now introduce a set of notation:

$$h_{n}(s):=(1-F_{0,n}(s))\prod_{k=1}^{n-1}F_{k,n}(s),$$
and for $i\leq n$ 
\begin{equation*}
a_{i,n}:=e^{S_{i}-S_{n}},\quad
b_{i,n}:=\sum_{k=i}^{n-1}e^{S_{i}-S_{k}},\quad 
a_{n}:=a_{0,n}=e^{-S_{n}}\quad \text{and} \quad b_{n}:=b_{0,n}=\sum_{k=0}^{n-1}e^{-S_{k}}.
\end{equation*}

We have the following equality:

\begin{lemma}
\label{L_represent2}Under Hypothesis $A1$%
$$
h_{n}(s) =\frac{1}{a_{n}\left( 1-s\right) ^{-1}+b_{n}}\frac{a_{n}\left(
1-s\right) ^{-1}}{a_{n}\left( 1-s\right) ^{-1}+b_{n}-b_{1}}.$$
\end{lemma}

\begin{proof}
Hypothesis A1 implies 
\begin{equation}
\ F_{i}(s)=\frac{q_{i}}{1-p_{i}s}=\frac{1}{1+e^{X_{i}}\left( 1-s\right) }
\end{equation}%
for all $i\in \mathbb{N}$. By induction we can prove that 
\begin{equation}
F_{0,n}(s)=1-\frac{1}{a_{n}\left( 1-s\right) ^{-1}+b_{n}}  \label{Frac_F}
\end{equation}%
and, therefore,%
\begin{eqnarray}  \label{expr_Fin}
F_{i,n}(s) &=&1-\frac{1}{a_{i,n}\left( 1-s\right) ^{-1}+b_{i,n}}  \notag \\
&=&1-\frac{a_{i}}{a_{n}\left( 1-s\right) ^{-1}+b_{n}-b_{i}}=\frac{%
a_{n}\left( 1-s\right) ^{-1}+b_{n}-b_{i+1}}{a_{n}\left( 1-s\right)
^{-1}+b_{n}-b_{i}}.
\end{eqnarray}%
Thus,%
\begin{eqnarray*}
h_{n}(s) &=&\frac{1}{a_{n}\left( 1-s\right) ^{-1}+b_{n}}\prod_{j=1}^{n-1}%
\frac{a_{n}\left( 1-s\right) ^{-1}+b_{n}-b_{j+1}}{a_{n}\left( 1-s\right)
^{-1}+b_{n}-b_{j}} \\
&=&\frac{1}{a_{n}\left( 1-s\right) ^{-1}+b_{n}}\frac{a_{n}\left( 1-s\right)
^{-1}}{a_{n}\left( 1-s\right) ^{-1}+b_{n}-b_{1}}.
\end{eqnarray*}%
This ends the proof.
\end{proof}

To end this section, we will provide an expression in terms of $a_i$'s and $%
b_i$'s for the random variable
\begin{equation*}
\mathcal{H}_{i,n}:=\left( 1-F_{i,n}(0)\right) \prod_{k\neq i}^{n-1}F_{k,n}(0).
\end{equation*}

\begin{corollary}
\label{C_fractional} Under Hypothesis $A1$ 
\begin{equation*}
\mathcal{H}_{0,n}=\frac{1}{a_{n}+b_{n}}\frac{a_{n}}{a_{n}+b_{n}-b_{1}}
\end{equation*}%
and, for any $i=1,2,...,n-1$ 
\begin{equation*}
\mathcal{H}_{i,n}=\frac{a_{i}}{a_{n}+b_{n}-b_{i+1}}\frac{a_{n}}{%
a_{n}+b_{n}-b_{1}}.
\end{equation*}
\end{corollary}

\begin{proof}
The first part of the Corollary is a direct consequence of Lemma \ref%
{L_represent2}, as $\mathcal{H}_{0,n}=h_{n}(0).$ The second part derives
from \eqref{expr_Fin}, by taking $s=0$.
\end{proof}

The next statement is a particular case of a theorem established in \cite%
{guivarc2001proprietes}.

\begin{lemma}
\label{L_Guiv0} Let $\eta _{n},n\geq 1$ be a sequence of i.i.d. positive random
variables having non-lattice distribution. Set $\Upsilon
_{n}:=\prod_{k=1}^{n}\eta _{k}$, $\Lambda _{n}:=\sum_{k=1}^{n}\Upsilon _{k}$. 
Assume that there exist strictly positive numbers $\varepsilon ,\alpha
,\beta $ and two nonegative continuous functions $g$ and $h$ defined on $%
[0,\infty )$ and not identically equal to zero and a constant $C>0$ such
that for all $a>0,c_{1}\geq 0$ and $c_{2}\geq 0$%
\begin{equation*}
g(a)\leq Ca^{\alpha },\ h(c_1)\leq \frac{C}{\left( 1+c_1\right) ^{\beta }},\
\left\vert h(c_{1})-h(c_{2})\right\vert \leq C\left\vert
c_{1}-c_{2}\right\vert ^{\varepsilon }.
\end{equation*}%
If%
\begin{equation*}
\mathbf{E}\log \eta _{1}=0,\ \mathbf{E}\left[ \eta _{1}^{\alpha }+\eta
_{1}^{-\varepsilon }\right] <\infty 
\end{equation*}%
then there exist two positive constants $c(\varphi ,\psi )$ and $c(\psi )$
such that 
\begin{equation}
\lim_{n\rightarrow \infty }n^{3/2}\mathbf{E}\left[ g(\Upsilon _{n})h(\Lambda
_{n})\right] =c(g,h),\ \ \lim_{n\rightarrow \infty }n^{1/2}\mathbf{E}\left[
h(\Lambda_{n})\right] =c(h).  \label{GuivStatement}
\end{equation}
\end{lemma}

\section{The case $i=n-N$} \label{sec_nN}

Thanks to the auxiliary results derived in the previous section, we have now the needed ingredients to prove the second statement of Theorem \ref{T_total}.

\begin{lemma}
\label{L_Largei}If Hypotheses A1-A2 are valid then, for each fixed $N$%
\begin{equation*}
\lim_{n\rightarrow \infty }\sqrt{n}\mathbf{P}\left( \mathcal{A}%
_{n-N}(n)\right) =r_{N}\in \left( 0,\infty \right) .
\end{equation*}
\end{lemma}

\begin{proof}
Taking the expectation with respect to the $\sigma $-algebra generated by
the sequence $F_{n-N},F_{n-N+1},...,F_{n}$ and making the changes $%
F_{j}\rightarrow \tilde{F}_{j-i}$ we write 
\begin{align*}
\mathbf{P}\left( \mathcal{A}_{i}(n)\right) &=\mathbf{E}\left[
(1-F_{i,n}(0))\prod_{k=i+1}^{n-1}F_{k,n}(0)%
\prod_{k=1}^{i-1}F_{k,i}(F_{i,n}(0))\right] \\
&=\mathbf{E}\left[ (1-\tilde{F}_{0,N}(0))\prod_{k=1}^{N-1}\tilde{F}%
_{k,N}(0)\Psi _{i}(\tilde{F}_{0,N}(0))\right] ,
\end{align*}%
where 
\begin{equation*}
\Psi _{i}(s):=\mathbf{E}\left[ \prod_{k=1}^{i-1}F_{k,i}(s)\right] .
\end{equation*}

Observe that as the environments are i.i.d.,
\begin{equation*}
\mathbf{E}\left[ \prod_{k=1}^{i-1}F_{k,i}(s)\right] =\mathbf{E}\left[
\prod_{k=1}^{i-1}F_{k,0}(s)\right] .
\end{equation*}%
Now from Lemma 2 in \cite{DLVZ19}, 
\begin{equation*}
\prod_{k=1}^{i-1}F_{k,0}(s)=\frac{(1-s)^{-1}}{(1-s)^{-1}+%
\sum_{k=1}^{i-1}e^{-S_{k}}}.
\end{equation*}%
By a direct application of the second statement in Lemma \ref{L_Guiv0}, we
obtain that for any $s\in \lbrack 0,1)$, 
\begin{equation*}
\lim_{i\rightarrow \infty }\sqrt{i}\Psi _{i}(s):=\psi (s)\in (0,\infty ).
\end{equation*}%
Hence, by the dominated convergence theorem we conclude that 
\begin{equation*}
\lim_{n\rightarrow \infty }\sqrt{n}\mathbf{P}\left( \mathcal{A}%
_{n-N}(n)\right) =\mathbf{E}\left[ (1-\tilde{F}_{0,N}(0))\prod_{k=1}^{N-1}%
\tilde{F}_{k,N}(0)\psi (\tilde{F}_{0,N}(0))\right] .
\end{equation*}

This ends the proof of Lemma \ref{L_Largei} and justifies the second statement of Theorem 1.
\end{proof}

\section{The case of fixed $i$} \label{sec_fixed_i}

We now consider the case of a fixed $i$. Introduce the running maximum and
minimum of the associated random walk $\mathbf{S}$  
\begin{equation} \label{def_LM}
M_{n}:=\max \left( S_{1},...,S_{n}\right) ,\quad L_{n}:=\min \left(
S_{0},S_{1},...,S_{n}\right) 
\end{equation}%
and denote by 
\begin{equation}
\tau (n):=\min \{0\leq k\leq n:S_{k}=L_{n}\}  \label{def_tau(n)}
\end{equation}
the moment of the first random walk minimum up to time $n$.

Observe that by a Sparre-Andersen identity (see for instance \cite{GV2017} p. 68) and according to Proposition 2.1 in \cite{ABKV} there exist positive
constants $K_{1}$ and $K_{2}$ such that, as $n\rightarrow \infty $ 
\begin{equation}
\mathbf{E}\left[ e^{S_{n}};\tau (n)=n\right] =\mathbf{E}\left[
e^{S_{n}};M_{n}<0\right] \sim \frac{K_{1}}{n^{3/2}},\quad \mathbf{E}\left[
e^{-S_{n}};L_{n}\geq 0\right] \sim \frac{K_{2}}{n^{3/2}}.
\label{AsymConditional}
\end{equation}

It will allow us to prove the following result, which will be the main tool for proving the first statement of Theorem \ref{T_total}.

\begin{lemma}
\label{L_Gui2}Under Hypotheses A1-A2 for any $x\geq 0$%
\begin{equation*}
\lim_{n\rightarrow \infty }n^{3/2}\mathbf{E}\left[ \frac{1}{a_{n}+b_{n}}%
\frac{a_{n}}{x+a_{n}+b_{n}}\right] =:\Pi (x)>0.
\end{equation*}%
In addition, there exists a constant $C\in (0,\infty )$ such that, for all $%
x>0$ 
\begin{equation*}
\mathbf{E}\left[ \frac{1}{a_{n}+b_{n}}\frac{a_{n}}{x+a_{n}+b_{n}}\right]
\leq \frac{C}{n^{3/2}}.
\end{equation*}
\end{lemma}

\begin{proof} The first statement follows from the first statement of Lemma \ref{L_Guiv0} with the functions $g(y)=y$ and $h(y)= 1/((1+y)(1+x+y))$.

To prove the second one, observe that
\begin{equation*}
\mathbf{E}\left[ \frac{1}{a_{n}+b_{n}}\frac{a_{n}}{x+a_{n}+b_{n}}\right]
\leq \mathbf{E}\left[ e^{-S_{n}}e^{2S_{\tau (n)}}\right] .
\end{equation*}%
Using (\ref{AsymConditional}), the right-hand side can be evaluated from above 
\begin{eqnarray*}
\mathbf{E}\left[ e^{-S_{n}}e^{2S_{\tau (n)}}\right]  &=&\sum_{k=0}^{n}%
\mathbf{E}\left[ e^{S_{k}-S_{n}}e^{S_{k}};\tau (n)=k\right]  \\
&=&\sum_{k=0}^{n}\mathbf{E}\left[ e^{S_{k}};\tau (k)=k\right] \mathbf{E}%
\left[ e^{S_{n-k}};L_{n-k}\geq 0\right]  \\
&\leq &C\sum_{k=0}^{n}\frac{1}{\left( k+1\right) ^{3/2}}\frac{1}{\left(
n-k+1\right) ^{3/2}}\leq \frac{C}{n^{3/2}},
\end{eqnarray*}%
as desired.
\end{proof} 

We can now prove the first statement of Theorem \ref{T_total}. By definition
\begin{eqnarray*}
\mathcal{H}_{i,n} &=&\frac{e^{-S_{i}}}{\sum_{k=i}^{n}e^{-S_{k}}}\frac{%
e^{-S_{n}}}{\sum_{k=1}^{n}e^{-S_{k}}} \\
&\mathcal{=}&\frac{1}{1+\sum_{k=i+1}^{n}e^{S_{i}-S_{k}}}\frac{%
e^{-(S_{n}-S_{i})}}{1+\sum_{k=1}^{i-1}e^{S_{i}-S_{k}}+%
\sum_{k=i+1}^{n}e^{S_{i}-S_{k}}}.
\end{eqnarray*}%
Taking the expectation with respect to the sequence $S_{0},S_{1},...,S_{i}$
we have 
\begin{align*}
\mathbf{P}\left( \mathcal{A}_{i}(n)\right) &=\mathbf{E}\left[ \frac{a_{i}}{%
a_{n}+b_{n}-b_{i+1}}\frac{a_{n}}{a_{n}+b_{n}-b_{1}}\right] \\
&=\mathbf{E}\left[ T_{n-i}\left(
\sum_{k=1}^{i-1}e^{S_{i}-S_{k}}\right) \right],
\end{align*}%
where for $m \in \N$,
\begin{equation*}
T_{m}(x):=\mathbf{E}\left[ e^{-S_{m}}\frac{1}{1+\Lambda_{m}}\frac{1}{%
1+x+\Lambda_{m}}\right] .
\end{equation*}%
Now applying Lemma \ref{L_Gui2} and using the dominated convergence theorem
we conclude that%
\begin{eqnarray*}
\lim_{n\rightarrow \infty }n^{3/2}\mathbf{P}\left( \mathcal{A}_{i}(n)\right)
&=&\mathbf{E}\left[ \lim_{n\rightarrow \infty
}n^{3/2}T_{n-i}\left( \sum_{k=1}^{i-1}e^{S_{i}-S_{k}}\right) \right] \\
&=&\mathbf{E}\left[ \Pi \left(
\sum_{k=1}^{i-1}e^{S_{i}-S_{k}}\right) \right] =:w_{i}\in \left( 0,\infty
\right) .
\end{eqnarray*}

\section{ The case $min$ ($i,n-i)\rightarrow \infty $} \label{sec_ini_to_infty}

The proof of the third statement of Theorem \ref{T_total} requires a different approach which we develop
in the next subsection.

\subsection{Measures $\mathbf{P}_{x}^{+}$ and $\mathbf{P}_{x}^{-}$}

As mentioned before, the survival of a clan is intimately related to the value of the running minimum of the associated random walk $\mathbf{S}$ initiated at 
its immigration time. More generally, conditioning on the event that the running minimum is not too small turned to be a very powerful tool in the study of the survival 
probability of BPRE's (see \cite{4h} for instance). We will use a similar technique, and to this aim
we need to perform two changes of measure using the
right-continuous functions $U:\mathbb{R}$ $\rightarrow \lbrack 0,\infty )$
and $V:\mathbb{R}$ $\rightarrow \lbrack 0,\infty )$ specified by
\begin{equation*}
U(x):=1+\sum_{n=1}^{\infty }\mathbf{P}\left( S_{n}\geq -x,M_{n}<0\right)
,\quad x\geq 0,
\end{equation*}
\begin{equation*}
V(x):=1+\sum_{n=1}^{\infty }\mathbf{P}\left( S_{n}<-x,L_{n}\geq 0\right)
,\quad x\leq 0.
\end{equation*}

It is known (see, for instance, \cite{4h}\ and \cite{ABKV}) that for any
oscillating random walk 
\begin{equation}
\mathbf{E}\left[ U(x+X);X+x\geq 0\right] =U(x),\quad x\geq 0,  \label{Mes1}
\end{equation}%
and
\begin{equation}
\mathbf{E}\left[ V(x+X);X+x<0\right] =V(x),\quad x\leq 0.  \label{Mes2}
\end{equation}

Moreover, we have the following asymptotics (see Lemma 2.1 in \cite{4h}): there exist two constants $C_1$ and $C_2$ such that for every $x \geq 0$, $m \in \N$
\begin{equation} \label{Asym1}
      \mathbf{P}(L_m \geq -x) \leq C_1 U(x) \mathbf{P}(L_m \geq 0) \leq U(x)\frac{C_2}{\sqrt{m}};
     \end{equation}
and according to Corollary 3 in \cite{Don12} there exists a constant $C_3$ such that as $m\rightarrow \infty $ 
\begin{equation}
\mathbf{P}\left( L_{m}\geq -x\right) \sim U(x)\mathbf{P}\left( L_{m}\geq
0\right) \sim U(x)\frac{C_3}{\sqrt{m}}  \label{Asym0}
\end{equation}%
uniformly in $0\leq x\ll \sqrt{m}$.

Let $\mathcal{E}=\left\{ F_{n},n\in \N\right\} $ be a random environment and
let $\mathcal{F}_{n},n\in \N,$ be the $\sigma $-field of events generated by
the random variables $F_{1},F_{2},...,F_{n}$ and the sequence $%
Y_{0},Y_{1},...,Y_{n}$. The set of these $\sigma $-fields forms a filtration 
$\mathfrak{F}$. The increment $X_{n},n\in \N,$ of the random walk $%
\mathbf{S}$ are measurable with respect to the $\sigma $-field $\mathcal{F}%
_{n}$. Using the martingale property (\ref{Mes1})-(\ref{Mes2}) of $U,V$ we
introduce in now a standard way (see, for instance, \cite{GV2017}, Chapter
7) a sequence of probability measures $\{ \mathbf{P}_{(n)}^{+},n\geq 1\} $
on the $\sigma $-field $\mathcal{F}_{n}$ by means of the densities 
\begin{equation*}
d\mathbf{P}_{(n)}^{+}:=U(S_{n})I\left\{ L_{n}\geq 0\right\} d\mathbf{P}.
\end{equation*}%
This and Kolmogorov's extension theorem show that, on a suitable probability
space there exists a probability measure $\mathbf{P}^{+}$ on the $\sigma $%
-field $\mathfrak{F}$ such that 
\begin{equation}
\mathbf{P}^{+}|\mathcal{F}_{n}=\mathbf{P}_{(n)}^{+},\ n\in \N.
\label{DefMeasures}
\end{equation}

In the sequel we allow for arbitrary initial value $S_{0}=x$. Then, we write 
$\mathbf{P}_{x}$ and $\mathbf{E}_{x}$ for the corresponding probability
measures and expectations. Thus, $\mathbf{P}=\mathbf{P}_{0}.$ With this
notation, (\ref{DefMeasures}) may be rewritten as follows: for every $%
\mathcal{F}_{n}$-measurable random variable $O_{n}$ 
\begin{equation*}
\mathbf{E}_{x}^{+}\left[ O_{n}\right] :=\frac{1}{U(x)}\mathbf{E}_{x}\left[
O_{n}U(S_{n});L_{n}\geq 0\right] ,\ x\geq 0.
\end{equation*}

Similarly, $V$ gives rise to probability measures $\mathbf{P}_{x}^{-},x\leq
0 $, which can be defined via: 
\begin{equation*}
\mathbf{E}_{x}^{-}\left[ O_{n}\right] :=\frac{1}{V(x)}\mathbf{E}_{x}\left[
O_{n}V(S_{n});M_{n}<0\right] ,\ x\leq 0.
\end{equation*}

By means of the measures $\mathbf{P}_{x}^{+}$, $\mathbf{P}_{x}^{-}$, we
investigate the limit behavior of certain conditional distributions.

For $\lambda >0$, let $\mu _{\lambda }$ and $\nu _{\lambda }$ be the
probability measures on $[0,+\infty )$ and $(-\infty ,0)$ given by their
densities 
\begin{equation*}
\mu _{\lambda }(dz)\ :=\ c_{1}e^{-\lambda z}U(z)1_{\{z\geq 0\}}\,dz\ ,\quad
\nu _{\lambda }(dz)\ :=\ c_{2}e^{\lambda z}V(z)1_{\{z<0\}}\,dz
\end{equation*}%
with 
\begin{equation}  \label{def_ci}
c_{1}^{-1}=c_{1\lambda }^{-1}:=\int_{0}^{\infty }e^{-\lambda
z}U(z)\,dz,\quad c_{2}^{-1}=c_{2\lambda }^{-1}:=\int_{-\infty
}^{0}e^{\lambda z}V(z)\,dz.
\end{equation}

The next two lemmas are natural modifications of Lemmas 7.3 and 7.5
in \cite{GV2017}, Chapter 7. We use the agreement $\delta n:=\lfloor \delta
n\rfloor $ for $0<\delta <1$ in their formulations.

\begin{lemma}
\label{L_Newp41} Take $0<\delta <1$. Let $G_{n}:=g_{n}(F_{1},\ldots
,F_{\delta n})$, $n \in \N$, be random variables with values in an Euclidean
(or polish) space $\mathcal{G}$ such that, as $n \to \infty$
\begin{equation*}
G_{n}\ \rightarrow \ G_{\infty }\quad \mathbf{P}^{+}\text{-a.s.}
\end{equation*}%
for some $\mathcal{G}$-valued random variable $G_{\infty }$. Also let $%
H_{n}:=h_{n}(F_{1},\ldots ,F_{\delta n})$, $n\geq 1$, be random variables
with values in an Euclidean (or polish) space $\mathcal{H}$ such that, as $n \to \infty$,
\begin{equation*}
H_{n}\ \rightarrow \ H_{\infty }\quad \mathbf{P}_{x}^{-}\text{-a.s. }
\end{equation*}%
for all $x\leq 0$ and some $\mathcal{H}$-valued random variable $H_{\infty }$. Denote 
\begin{equation*}
\tilde{H}_{n}:=h_{n}(F_{n},\ldots ,F_{n-\delta n+1})\ .
\end{equation*}%
Let, further $\psi (z),z\geq 0$, be a nonnegative continuous function such
that $\psi (z)e^{-\theta z},z\geq 0$, is bounded for some $\theta >0$. Then,
for any bounded continuous function $\varphi :\mathcal{G}\times \mathcal{H}%
\times \mathbb{R}\rightarrow \mathbb{R}$, and $\lambda >\theta $ 
\begin{align}
 \lim_{n\rightarrow \infty }&\frac{\mathbf{E}[\psi (S_{n})\varphi (G_{n},%
\tilde{H}_{n},S_{n})e^{-\lambda S_{n}}\;;\;L_{n}\geq 0]}{\mathbf{E}%
[e^{-\lambda S_{n}};L_{n}\geq 0]\ }\;  \notag \\
& =\ \iiint \psi (-z)\varphi (u,v,-z)\mathbf{P}^{+}\left( G_{\infty }\in
du\right) \mathbf{P}_{z}^{-}\left( H_{\infty }\in dv\right) \nu _{\lambda
}(dz)\ .  \label{Cond1}
\end{align}
\end{lemma}

The following lemma is a counterpart.

\begin{lemma}
\label{L_Newp42} Let $G_{n},H_{n},\tilde{H}_{n}$, $ n \in \N $ be
as in Lemma \ref{L_Newp41}, now fulfilling, as $n \to \infty$
\begin{equation*}
G_{n}\ \rightarrow \ G_{\infty }\quad \mathbf{P}_{x}^{+}\text{-a.s.}
,\quad H_{n}\ \rightarrow \ H_{\infty }\quad \mathbf{P}^{-}\text{-a.s.}
\end{equation*}%
for all $x\geq 0$. Let, further $\psi (z),z\geq 0$, be a nonnegative
continuous function such that $\psi (z)e^{\theta z},z<0$, is bounded for some 
$\theta >0$. Then, for any bounded continuous function $\varphi :\mathcal{G}%
\times \mathcal{H}\times \mathbb{R}\rightarrow \mathbb{R}$ and for $\lambda
>\theta $ 
\begin{align}
& \frac{\mathbf{E}[\psi (S_{n})\varphi (G_{n},\tilde{H}_{n},S_{n})e^{\lambda
S_{n}}\;;\;\tau (n)=n]}{\mathbf{E}[e^{\lambda S_{n}};\tau(n)=n]}\;  \notag
\\
& \rightarrow \ \iiint \psi (-z)\varphi (u,v,-z)\mathbf{P}_{z}^{+}\left(
G_{\infty }\in du\right) \mathbf{P}^{-}\left( H_{\infty }\in dv\right) \mu
_{\lambda }(dz)\ .  \label{Cond3}
\end{align}
\end{lemma}

\begin{proof}
The proofs of the two statements are very similar. We show only the first
one. Write $\lambda =\theta +\delta ,$ $\delta >0$. Since $\psi
(z)e^{-\theta z}\varphi (x,y,z)$ is a bounded continuous function, we may
apply Lemma 7.3 in \cite{GV2017}, Chapter 7 and using the definition of $\nu
_{\lambda }(dz)$ conclude that, as $n\rightarrow \infty $

\begin{multline*}
\frac{\mathbf{E}[\psi (S_{n})\varphi (G_{n},\tilde{H}_{n},S_{n})e^{-\lambda
S_{n}}\;;\;L_{n}\geq 0]}{\mathbf{E}[e^{-\lambda S_{n}};L_{n}\geq 0]} \\
=\frac{\mathbf{E}[\psi (S_{n})e^{-\theta S_{n}}\varphi (G_{n},\tilde{H}%
_{n},S_{n})e^{-\delta S_{n}}\;;\;L_{n}\geq 0]}{\mathbf{E}[e^{-\delta
S_{n}};L_{n}\geq 0]}\frac{\mathbf{E}[e^{-\delta S_{n}};L_{n}\geq 0]}{\mathbf{%
E}[e^{-\lambda S_{n}};L_{n}\geq 0]} \\
\rightarrow \ \frac{\iiint e^{\theta z}\psi (-z)\varphi (u,v,-z)\mathbf{P}%
^{+}\left( G_{\infty }\in du\right) \mathbf{P}_{z}^{-}\left( H_{\infty }\in
dv\right) \nu _{\delta }(dz)}{\int_{-\infty }^{0}e^{\theta z}\nu _{\delta
}(dz)} \\
=\frac{\iiint e^{\theta z}\psi (-z)\varphi (u,v,-z)\mathbf{P}^{+}\left(
G_{\infty }\in du\right) \mathbf{P}_{z}^{-}\left( H_{\infty }\in dv\right)
e^{\delta z}V(z)\,dz}{\int_{-\infty }^{0}e^{\delta z}V(z)\,dz}\frac{%
\int_{-\infty }^{0}e^{\delta z}V(z)\,dz}{\int_{-\infty }^{0}e^{(\theta
+\delta )z}V(z)\,dz} \\
=\iiint \psi (-z)\varphi (u,v,-z)\mathbf{P}^{+}\left( G_{\infty }\in
du\right) \mathbf{P}_{z}^{-}\left( H_{\infty }\in dv\right) \nu _{\lambda
}(dz),
\end{multline*}%
as desired.
\end{proof}

\begin{lemma}
\label{L_Condx}Let Hypotheses $A1-A3$ be valid and $O_{n}$ be a sequence of
uniformly bounded random variables adapted to the filtration $\mathfrak{F}%
=\left( \mathcal{F}_{n},n\geq 1\right) $ such that, as $n\rightarrow \infty $%
\begin{equation*}
O_{n}\rightarrow O_{\infty }\text{ \ \ \ }\mathbf{P}_{x}^{+}\text{ a.s.}
\end{equation*}%
for all $x$ from a finite interval $[0,N]$. Then, as $n\rightarrow \infty $ 
\begin{equation}
\mathbf{E}\left[ O_{n}|L_{n}\geq -x\right] \rightarrow \mathbf{E}_{x}^{+}%
\left[ O_{\infty }\right]  \label{Cond+}
\end{equation}%
uniformy in $x\in \lbrack 0,N].$

Let 
\begin{equation*}
O_{n}\rightarrow O_{\infty }\text{ \ \ \ }\mathbf{P}_{-x}^{-}\text{ a.s.}
\end{equation*}%
for all $x$ from a finite interval $[0,N]$. Then, as $n\rightarrow \infty $ 
\begin{equation}
\mathbf{E}\left[ O_{n}|M_{n}<-x\right] \rightarrow \mathbf{E}_{-x}^{-}\left[
O_{\infty }\right]  \label{Cond-}
\end{equation}%
uniformly in $x\in \lbrack 0,N].$
\end{lemma}

\begin{proof} We prove (\ref{Cond+}). Assume without loss of generality
that $\left\vert O_{n}\right\vert \leq 1$ for all $n$. First we show that,
as $n\rightarrow \infty $ 
\begin{equation*}
\mathbf{E}\left[ O_{k}|L_{n}\geq -x\right] \rightarrow \mathbf{E}_{x}^{+}%
\left[ O_{k}\right] 
\end{equation*}%
for any fixed $k$ and $x\in \lbrack 0,N].$ To this aim we write%
\begin{equation*}
\mathbf{E}\left[ O_{k}|L_{n}\geq -x\right] =\frac{\mathbf{E}\left[
O_{k};L_{n}\geq -x\right] }{\mathbf{P}(L_{n}\geq -x)}=\frac{\mathbf{E}\left[
O_{k}\mathbf{P}\left( L_{n-k}^{\prime }\geq -x-S_{k}|S_{k}\right) ;L_{k}\geq
-x\right] }{\mathbf{P}(L_{n}\geq -x)}.
\end{equation*}%
In view of \eqref{Asym0} 
\begin{equation*}
\frac{\mathbf{P}(L_{n}\geq -y)}{\mathbf{P}(L_{n}\geq -x)}\sim \frac{U(y)}{%
U(x)}
\end{equation*}%
as $n\rightarrow \infty $ uniformly in $x\in \lbrack 0,N]$ and $y=o\left( 
\sqrt{n}\right) $. Besides, using \eqref{Asym1} and since
\begin{eqnarray*}
\frac{1}{U(x)}\mathbf{E}\left[ O_{k}U\left( x+S_{k}\right) ;L_{k}\geq -x%
\right] &=&\frac{1}{U(x)}\mathbf{E}_{x}\left[ O_{k}U\left( S_{k}\right)
;L_{k}\geq 0\right] \\
&=&\mathbf{E}_{x}^{+}\left[ O_{k}\right] <\infty ,
\end{eqnarray*}%
the dominated convergence theorem gives%
\begin{eqnarray*}
\lim_{n\rightarrow \infty }\mathbf{E}\left[ O_{k}|L_{n}\geq -x\right] &=&%
\mathbf{E}\left[ O_{k}\lim_{n\rightarrow \infty }\frac{\mathbf{P}\left(
L_{n-k}^{\prime }\geq -x-S_{k}|S_{k}\right) }{\mathbf{P}(L_{n}\geq -x)}%
;L_{k}\geq -x\right] \\
&=&\frac{1}{U(x)}\mathbf{E}\left[ O_{k}U\left( x+S_{k}\right) ;L_{k}\geq -x%
\right] =\mathbf{E}_{x}^{+}\left[ O_{k}\right]
\end{eqnarray*}%
for each fixed $x\in \lbrack 0,N]$, where $\mathbf{S}'$ is distributed as $\mathbf{S}$, the two random walks are independent, and $L'$ is the running minimum of $\mathbf{S}'$.

Further, we fix $\gamma >1$ and consider the difference%
\begin{eqnarray*}
\left\vert \mathbf{E}\left[ O_{n};L_{n}\geq -x\right] -\mathbf{E}\left[
O_{n};L_{n\gamma }\geq -x\right] \right\vert &\leq &\mathbf{P}(L_{n}\geq -x)-%
\mathbf{P}\left( L_{n\gamma }\geq -x\right) \\
&=&U(x)(1+\varepsilon (n,x,\gamma ))\left( \mathbf{P}(L_{n}\geq 0)-\mathbf{P}%
\left( L_{n\gamma }\geq 0\right) \right)
\end{eqnarray*}%
where, for any $\varepsilon >0$ 
\begin{equation*}
\sup_{x\in \lbrack 0,N]}\left\vert \varepsilon (n,x,\gamma )\right\vert \leq
\varepsilon
\end{equation*}%
for all sufficiently large $n$. Hence using \eqref{Asym0} it follows that%
\begin{equation}
\frac{\left\vert \mathbf{E}\left[ O_{n};L_{n}\geq -x\right] -\mathbf{E}\left[
O_{n};L_{n\gamma }\geq -x\right] \right\vert }{\mathbf{P}(L_{n}\geq -x)}\leq
(1+\varepsilon _{1}(n,x,\gamma ))\frac{\left( \mathbf{P}(L_{n}\geq 0)-%
\mathbf{P}\left( L_{n\gamma }\geq 0\right) \right) }{\mathbf{P}(L_{n}\geq 0)}
\label{Gamma_diff}
\end{equation}%
where, for any $\varepsilon >0$ 
\begin{equation*}
\sup_{x\in \lbrack 0,N]}\left\vert \varepsilon _{1}(n,x,\gamma )\right\vert
\leq \varepsilon
\end{equation*}%
for all sufficiently large $n$.
Further,%
\begin{eqnarray*}
\left\vert \mathbf{E}\left[ O_{n};L_{n\gamma }\geq -x\right] -\mathbf{E}%
\left[ O_{k};L_{n\gamma }\geq -x\right] \right\vert &\leq &\mathbf{E}\left[
\left\vert O_{n}-O_{k}\right\vert ;L_{n\gamma }\geq -x\right] \\
&=&\mathbf{E}\left[ \left\vert O_{n}-O_{k}\right\vert \mathbf{P}%
_{S_{n}}( L'_{n(1-\gamma )}\geq -x) ;L_{n}\geq -x\right] \\
&\leq &C\mathbf{E}\left[ \left\vert O_{n}-O_{k}\right\vert
U(x+S_{n});L_{n}\geq -x\right] \mathbf{P}( L'_{n(\gamma -1)}\geq 0)
\\
&=&CU(x)\mathbf{E}_{x}^{+}\left[ \left\vert O_{n}-O_{k}\right\vert \right] 
\mathbf{P}\left( L_{n(\gamma -1)}\geq 0\right) .
\end{eqnarray*}

Thus, letting first $n$ to infinity and then $k$ to infinity we conclude
that 
\begin{eqnarray*}
\lim_{k\rightarrow \infty }\lim_{n\rightarrow \infty }\frac{\left\vert 
\mathbf{E}\left[ O_{n};L_{n\gamma }\geq -x\right] -\mathbf{E}\left[
O_{k};L_{n\gamma }\geq -x\right] \right\vert }{\mathbf{P}\left( L_{n}\geq
-x\right) } \leq C\lim_{k\rightarrow \infty }\mathbf{E}_{x}^{+}\left[ \left\vert
O_{\infty }-O_{k}\right\vert \right] \frac{\mathbf{1}}{\sqrt{\gamma -1}}=0.
\end{eqnarray*}%
Thus,%
\begin{equation*}
\mathbf{E}\left[ O_{n}|L_{n}\geq -x\right] \rightarrow \mathbf{E}_{x}^{+}%
\left[ O_{\infty }\right] 
\end{equation*}%
for each fixed $x\in \lbrack 0,N].$

Note further, that for $x>y$%
\begin{align*}
\left\vert \frac{\mathbf{E}\left[ O_{n};L_{n}\geq -x\right] }{\mathbf{P}%
(L_{n}\geq -x)}-\frac{\mathbf{E}\left[ O_{n};L_{n}\geq -y\right] }{\mathbf{P}%
(L_{n}\geq -y)}\right\vert  
&\leq \left\vert \frac{\mathbf{E}\left[ O_{n};L_{n}\geq -x\right] -\mathbf{E%
}\left[ O_{n};L_{n}\geq -y\right] }{\mathbf{P}(L_{n}\geq -x)}\right\vert  \\
& \quad +\left\vert\mathbf{E}\left[ O_{n};L_{n}\geq -y\right] \right\vert \left( \frac{1}{\mathbf{P}%
(L_{n}\geq -y)}-\frac{1}{\mathbf{P}(L_{n}\geq -x)}\right)  \\
&\leq 2\left( 1-\frac{\mathbf{P}(L_{n}\geq -y)}{\mathbf{P}(L_{n}\geq -x)}%
\right) .
\end{align*}%
In view of Hypothesis A3 $\mathbf{P}(L_{n}\geq -x)$ is continuous in $x$ for
each fixed $n$. Therefore, the conditional expectation $\mathbf{E}\left[
O_{n}|L_{n}\geq -x\right] $ is continuous in $x$ for each fixed $n$. Since%
\begin{equation*}
\frac{\mathbf{P}(L_{n}\geq -y)}{\mathbf{P}(L_{n}\geq -x)}\rightarrow \frac{%
U(y)}{U(x)}
\end{equation*}%
as $n\rightarrow \infty $ uniformly in $x$ and $y$ from any finite interval,
it follows that, for any $\varepsilon >0$ 
\begin{equation*}
\left\vert \frac{\mathbf{E}\left[ O_{n};L_{n}\geq -x\right] }{\mathbf{P}%
(L_{n}\geq -x)}-\frac{\mathbf{E}\left[ O_{n};L_{n}\geq -y\right] }{\mathbf{P}%
(L_{n}\geq -y)}\right\vert \leq 2\left( 1-\left( 1-\varepsilon \right) \frac{%
U(y)}{U(x)}\right) 
\end{equation*}%
for all sufficiently large $n$. \ Therefore, $\mathbf{E}_{x}^{+}\left[
O_{\infty }\right] $ is continuous in $x$. This implies, in turn, that
convergence 
\begin{equation*}
\mathbf{E}\left[ O_{n}|L_{n}\geq -x\right] \rightarrow \mathbf{E}_{x}^{+}%
\left[ O_{\infty }\right] 
\end{equation*}%
as $n\rightarrow \infty $ is uniform in $x\in \left[ 0,N\right] $, proving (%
\ref{Cond+}).

The validity of (\ref{Cond-}) can be checked in a similar way.
\end{proof}

Lemma \ref{L_Condx} allows us to establish further modifications of Lemmas \ref{L_Newp41} and \ref{L_Newp42}.

\begin{lemma}
\label{L_pr4New} Take $0<\delta <1$. Let $G_{n}:=g_{n}(F_{1},\ldots
,F_{\delta n})$ , $n \in \N$, be random variables with values in an Euclidean
(or polish) space $\mathcal{G}$ such that, as $n \to \infty$
\begin{equation*}
G_{n}\ \rightarrow \ G_{\infty }\quad \mathbf{P}^{+}\text{-a.s.}
\end{equation*}%
for some $\mathcal{G}$-valued random variable $G_{\infty }$. Also let $%
H_{n}:=h_{n}(F_{1},\ldots ,F_{\delta n})$, $ n \in \N $, be random variables
with values in an Euclidean (or polish) space $\mathcal{H}$ such that, as $n \to \infty$
\begin{equation*}
H_{n}\ \rightarrow \ H_{\infty }\quad \mathbf{P}_{x}^{-}\text{-a.s. }
\end{equation*}%
for all $x\leq 0$ and some $\mathcal{H}$-valued random variable $H_{\infty }$. Denote 
\begin{equation*}
\tilde{H}_{n}:=h_{n}(F_{n},\ldots ,F_{n-\delta n+1}).
\end{equation*}%
Let $T_{n}:=t_{n}(F_{1},\ldots ,F_{n})$ be random variables with values in
an Euclidean (or polish) space $\mathcal{T}$ such that, as $n\rightarrow
\infty $ 
\begin{equation*}
T_{n}\ \rightarrow \ T_{\infty }\quad \mathbf{P}_{x}^{+}\text{-a.s. }
\end{equation*}%
for all $x\geq 0$ and some $\mathcal{T}$-valued random variable $T_{\infty }$. 
Denote for $r \leq n$ $\hat{T}_{n-r}:=t_{n-r}(F_{n-r+1},\ldots ,F_{n})$. Then, for $%
\lambda >0$ and for any bounded continuous function $\varphi :\mathcal{G}%
\times \mathcal{H}\times \mathbb{R\times }\mathcal{T}\rightarrow \mathbb{R}$%
, as $\min \left( r,n-r\right) \rightarrow \infty ,$ 
\begin{eqnarray*}
&&\frac{\mathbf{E}[\varphi (G_{r},\tilde{H}_{r},S_{r};\hat{T}%
_{n-r})e^{-\lambda S_{r}}\;;\;L_{n}\geq 0]}{\mathbf{E}[e^{-\lambda
S_{r}};L_{n}\geq 0]\ } \\
\; &\rightarrow &\ \iiiint U(-z)\varphi (u,v,-z,t)\mathbf{P}^{+}\left(
G_{\infty }\in du\right) \mathbf{P}_{z}^{-}\left( H_{\infty }\in dv\right) 
\mathbf{P}_{-z}^{+}\left( T_{\infty }\in dt\right) \nu _{\lambda }(dz)\ .
\end{eqnarray*}
\end{lemma}

\begin{proof}
 Let $N$ be fixed and 
\begin{equation} \label{defPsiN}
\Psi _{N}(x)=\left\{ 
\begin{array}{ccc}
0 & \text{if} & x<-1, \\ 
1+x & \text{if} & x\in (-1,0), \\ 
1 & \text{if} & x\in \lbrack 0,N], \\ 
N+1-x & \text{if} & x\in (N,N+1), \\ 
0 & \text{if} & x>N+1.%
\end{array}%
\right. 
\end{equation}
 
Applying Lemma \ref{L_Condx} and (\ref{Asym0}) we obtain
that, for any bounded continuous functions $\varphi _{1}:\mathcal{G}\times 
\mathcal{H}\times \mathbb{R}\rightarrow \mathbb{R}$ and $\varphi _{2}:%
\mathcal{T}\rightarrow \mathbb{R}$ as $n-r\rightarrow \infty $%
\begin{align*}
\mathbf{E}&\left[ e^{-\lambda S_{r}}\varphi _{1}(G_{r},\tilde{H}%
_{r},S_{r})\varphi _{2}(\hat{T}_{n-r})\Psi _{N}(S_{r});L_{n}\geq 0\right] \ 
\\
&=\mathbf{E}[e^{-\lambda S_{r}}\varphi _{1}(G_{r},\tilde{H}_{r},S_{r})\Psi
_{N}(S_{r})\mathbf{E}\left[ \varphi _{2}(T_{n-r}^{\prime });
L_{n-r}^{\prime }\geq -S_{r} |S_{r}\right] ;L_{r}\geq 0] \\
&\sim \mathbf{E}[e^{-\lambda S_{r}}\varphi _{1}(G_{r},\tilde{H}%
_{r},S_{r})U(S_{r})\Psi _{N}(S_{r}); L_{r}\geq 0;  \mathbf{E}%
_{S_{r}}^{+}[\varphi _{2}(T_{\infty })]]\mathbf{P}\left( L_{n-r}\geq
0\right),
\end{align*}%
where we used that $\Psi _{N}(S_{r})\mathbf{1}_{L_{r}\geq 0} =0$ if $S_r \notin [0,N+1]$.
Now recalling (\ref{Cond1}) we obtain as $r\rightarrow \infty $%
\begin{eqnarray*}
&&\mathbf{E}[e^{-\lambda S_{r}}\varphi _{1}(G_{r},\tilde{H}%
_{r},S_{r})U(S_{r})\Psi _{N}(S_{r});  L_{r}\geq 0;  \mathbf{E}%
_{S_{r}}^{+}[\varphi _{2}(T_{\infty })]] \\
&\sim &\mathbf{E}[e^{-\lambda S_{r}}; L_{r}\geq 0 ]\times  \\
&&\times \iiint U(-z)\Psi _{N}(-z)\mathbf{E}_{-z}^{+}[\varphi _{2}(T_{\infty
})]\varphi _{1}(u,v,-z)\mathbf{P}^{+}\left( G_{\infty }\in du\right) \mathbf{%
P}_{z}^{-}\left( H_{\infty }\in dv\right) \nu _{\lambda }(dz) \\
&=&\mathbf{E}[e^{-\lambda S_{r}}; L_{r}\geq 0 ]\times  \\
&&\times \iiiint U(-z)\Psi _{N}(-z)\varphi _{2}(t)\varphi _{1}(u,v,-z)%
\mathbf{P}^{+}\left( G_{\infty }\in du\right) \mathbf{P}_{z}^{-}\left(
H_{\infty }\in dv\right) \mathbf{P}_{-z}^{+}\left( T_{\infty }\in dt\right)
\nu _{\lambda }(dz)\ .
\end{eqnarray*}%
Using the same arguments as earlier we conclude that%
\begin{align*}
\lim_{\min \left( r,n-r\right) \rightarrow \infty }&\frac{\mathbf{E}%
[\varphi _{1}(G_{r},\tilde{H}_{r},S_{r})\varphi _{2}(\hat{T}%
_{n-r})e^{-\lambda S_{r}}\;;\;L_{n}\geq 0]}{\mathbf{E}[e^{-\lambda
S_{r}};L_{n}\geq 0]\ } \\
&=\iiiint U(-z)\varphi _{2}(t)\varphi _{1}(u,v,-z)\mathbf{P}^{+}\left(
G_{\infty }\in du\right) \mathbf{P}_{z}^{-}\left( H_{\infty }\in dv\right) 
\mathbf{P}_{-z}^{+}\left( T_{\infty }\in dt\right) \nu _{\lambda }(dz).
\end{align*}

This proves our statement for the case $\varphi =\varphi _{2}\times \varphi
_{1}$. The general case follows from the Weierstrass theorem on
approximation of multivariate bounded continuous functions by polynomials
and a standard truncation procedure.
\end{proof}

\begin{lemma}
\label{L_pr3New}Let $G_{r},H_{r},\tilde{H}_{r}, T_r, \hat{T}_r$, $r \in \N $, be as in Lemma \ref{L_pr4New}, now fulfilling, as $r \to \infty$
\begin{equation*}
G_{r}\ \rightarrow \ G_{\infty }\quad \mathbf{P}^{+}_x\text{-a.s.}\quad \forall x\geq 0 ,\quad
H_{r}\ \rightarrow \ H_{\infty }\quad \mathbf{P}^{-}\text{-a.s. and }%
T_{r}\rightarrow T_{\infty }\text{ }\mathbf{P}^{+}\text{-a.s.}
\end{equation*}%
for all $x\geq 0$. Then, for $\lambda >0$ and for any bounded continuous
function $\varphi :\mathcal{G}\times \mathcal{H}\times \mathbb{R\times }%
\mathcal{T}\rightarrow \mathbb{R}$, as $\min \left( r,n-r\right) \rightarrow
\infty $ 
\begin{eqnarray*}
&&\frac{\mathbf{E}[\varphi (G_{r},\tilde{H}_{r},S_{r};\hat{T}_{n-r})e^{\lambda
S_{r}}\;;\;\tau (n)=r]}{\mathbf{E}[e^{\lambda S_{r}};\tau (n)=r]} \\
&\rightarrow &\ \iiiint \varphi (u,v,-z,t)\,\mathbf{P}^{+}_z\left(
G_{\infty }\in du\right) \mathbf{P}^{-}\left( H_{\infty }\in dv\right) 
\mathbf{P}^{+}\left( T_{\infty }\in dt\right) \mu _{\lambda }(dz)\ .
\end{eqnarray*}
\end{lemma}

\begin{proof} We again consider bounded continuous functions $\varphi _{1}:%
\mathcal{G}\times \mathcal{H}\times \mathbb{R}\rightarrow \mathbb{R}$ and $%
\varphi _{2}:\mathcal{T}\rightarrow \mathbb{R}$. Then%
\begin{multline*}
\mathbf{E}\left[ \varphi _{1}(G_{r},\tilde{H}_{r},S_{r})\varphi
_{2}(\hat{T}_{n-r})e^{\lambda S_{r}}\;;\;\tau (n)=r\right]  \\
=\mathbf{E}\left[ \varphi _{1}(G_{r},\tilde{H}_{r},S_{r})e^{\lambda
S_{r}}\;;\;\tau (r)=r\right] \mathbf{E}\left[ \varphi
_{2}(T_{n-r});\;L_{n-r}\geq 0\right] .
\end{multline*}%
This representation, Lemma \ref{L_Newp42} with $\psi (t)\equiv 1$ and Lemma %
\ref{L_Condx} show that 
\begin{multline*}
\frac{\mathbf{E}\left[ \varphi _{1}(G_{r},\tilde{H}_{r},S_{r})\varphi
_{2}(T_{n-r})e^{\lambda S_{r}}\;;\;\tau (n)=r\right] }{\mathbf{E}[e^{\lambda
S_{r}};\tau (n)=r]} \\
\rightarrow \iiint \varphi _{1}(u,v,-z)\mathbf{P}_z^{+}\left( G_{\infty }\in
du\right) \mathbf{P}^{-}\left( H_{\infty }\in dv\right) \mu _{\lambda
}(dz)\mathbf{E}^{+}\left[ T_{\infty }\right]  \qquad \qquad \qquad \\
= \iiiint \varphi _{1}(u,v,-z)\varphi _{2}(t)\,\mathbf{P}_z^{+}\left(
G_{\infty }\in du\right) \mathbf{P}^{-}\left( H_{\infty }\in dv\right) 
\mathbf{P}^{+}\left( T_{\infty }\in dt\right) \mu _{\lambda }(dz).
\end{multline*}%
The general case follows by the same arguments as in the previous lemma.
\end{proof}

\begin{lemma}
\bigskip \label{L_pr6New}Let $G_{r},H_{r},\tilde{H}_{r}, T_r, \hat{T}_r$, $r \in \N$, be as in Lemma \ref{L_pr4New}, and $%
H_{r,N}=h_{r,N}(F_{N},F_{N+1},...,F_{r\delta -1})$, $\tilde{H}%
_{r,N}=h_{r,N}(F_{r-N},F_{r-N-1},...,F_{r-\delta n+1})$ now fulfilling as $%
n\rightarrow \infty $ 
\begin{equation*}
G_{r}\ \rightarrow \ G_{\infty }\quad \mathbf{P}_x^{+}\text{-a.s.},\quad \forall x \geq 0, \quad
H_{r}\ \rightarrow \ H_{\infty },\;H_{r,N}\rightarrow H_{\infty .N}\text{ }%
\mathbf{P}^{-}\text{-a.s.}
\end{equation*}%
and $T_{r}\rightarrow T_{\infty }$ $\mathbf{P}^{+}$-a.s. as $r\rightarrow
\infty $. Then, for $\lambda >0$ and for any bounded continuous function $%
\varphi :\mathcal{G}\times \mathcal{H}\times \mathcal{H}\times \mathbb{%
R\times }\mathcal{T}\rightarrow \mathbb{R}$, as $\min \left( r,n-r\right)
\rightarrow \infty $ 
\begin{eqnarray*}
&&\frac{\mathbf{E}[\varphi (G_{r},\tilde{H}_{r},\tilde{H}%
_{r,N},S_{r};\hat{T}_{n-r})e^{\lambda S_{r}}\;;\;\tau (n)=r]}{\mathbf{E}%
[e^{\lambda S_{r}};\tau (n)=r]} \\
&\rightarrow &\ \iiiint \varphi (u,v_1,v_2,-z,t)\,\mathbf{P}_z^{+}\left(
G_{\infty }\in du\right) \mathbf{P}^{-}\left( H_{\infty }\in
dv_{1},H_{\infty ,N}\in dv_{2}\right) \mathbf{P}^{+}\left( T_{\infty }\in
dt\right) \mu_{\lambda }(dz)\ .
\end{eqnarray*}
\end{lemma}

\begin{proof}
 The proof is similar as the previous ones, and we do not provide it.
\end{proof}

\subsection{Some properties of driftless random walks}

In this subsection we consider a driftless random walk 
\begin{equation*}
S_{0}=0,\quad S_{k}=X_{1}+...+X_{k},\quad k \in \N
\end{equation*}%
satisfying the conditions%
\begin{equation}
\mathbf{E}X_{i}=0,\quad \sigma ^{2}=VarX_{i}\in \left( 0,\infty \right) .
\label{CondRW}
\end{equation}

\begin{lemma}
\label{L_simpleconditional}If condition (\ref{CondRW}) is valid then for
any $\lambda >0$ there exist positive constants $D_{i}=D_{i}(\lambda
),i=1,2,3$ such that, 
\begin{align}
\lim_{\min (r,n-r)\rightarrow \infty }r^{3/2}\left( n-r\right) ^{1/2}\mathbf{%
E}[e^{-\lambda S_{r}};L_{n} &\geq 0]=D_{1},  \label{D_1} \\
\lim_{\min (r,n-r)\rightarrow \infty }r^{3/2}\left( n-r\right) ^{1/2}\mathbf{%
E}[e^{\lambda S_{r}};M_{n} &<0]=D_{2},  \label{D_2} \\
\lim_{\min (r,n-r)\rightarrow \infty }r^{3/2}\left( n-r\right) ^{1/2}\mathbf{%
E}[e^{\lambda S_{r}};\tau (n) &=r]=D_{3}.  \label{D_3}
\end{align}
\end{lemma}

\begin{proof} Let us prove relation (\ref{D_1}). For a fixed $r\in N$, let 
$S_{j}^{\prime }:=S_{j+r}-S_{r},j=0,1,2,...$ and $%
L_{m}^{\prime }:=\min_{0\leq j\leq m}S_{j}^{\prime }.$ From the Markov property,
\begin{equation*}
\mathbf{E}[e^{-\lambda S_{r}};L_{n}\geq 0]\ =\mathbf{E}[e^{-\lambda S_{r}}%
\mathbf{P}\left( L_{n-r}^{\prime }\geq -S_{r}|S_{r}\right) ;L_{r}\geq 0]\ .
\end{equation*}%

Recall the definition of $\Psi_N$ in \eqref{defPsiN}.
Using \eqref{Asym0} and observing
that $U(x)e^{-\theta x},x\geq 0$ is bounded for any $\theta >0$ we conclude
by Lemma \ref{L_Newp41} with $\varphi (x)\equiv 1$ that, for any $\lambda >0$
as $\min (r,n-r)\rightarrow \infty $ 
\begin{multline*}
\mathbf{E}\left[ e^{-\lambda S_{r}}\mathbf{P}\left( L_{n-r}^{\prime }\geq
-S_{r}|S_{r}\right) \Psi _{N}(S_{r});L_{r}\geq 0\right]  \\
\sim \mathbf{E}[e^{-\lambda S_{r}}U(S_{r})\Psi _{N}(S_{r});L_{r}\geq 0]%
\mathbf{P}\left( L_{n-r}\geq 0\right)  \\
\sim \mathbf{E}[e^{-\lambda S_{r}};L_{r}\geq 0]\mathbf{P}\left(
L_{n-r}\geq 0\right) \times \int U(-z)\Psi _{N}(-z)v_{\lambda }(dz).
\end{multline*}
Note further that 
\begin{align*}
\mathbf{E}[e^{-\lambda S_{r}}U(S_{r});S_{r} \leq N,L_{r}\geq 0]&\leq 
\mathbf{E}[e^{-\lambda S_{r}}U(S_{r})\Psi _{N}(S_{r});L_{r}\geq 0] \\
&\leq \mathbf{E}[e^{-\lambda S_{r}}U(S_{r});S_{r}\leq N+1;L_{r}\geq 0].
\end{align*}%
According to Theorem 4 in \cite{VW07} there exists a constant $C>0$ such
that for any $\Delta >0$, as $r\rightarrow \infty $ 
\begin{equation*}
\mathbf{P}\left( S_{r}\in (x,x+\Delta ];L_{r}\geq 0\right) \sim \frac{C}{%
r^{3/2}}\int_{x}^{x+\Delta }V(-w)dw
\end{equation*}%
uniformly in $0\leq x\ll \sqrt{r}$. Hence it follows that, as $r\rightarrow
\infty $%
\begin{equation*}
\mathbf{E}[e^{-\lambda S_{r}}U(S_{r});S_{r}\leq N;L_{r}\geq 0]\sim \frac{C}{%
r^{3/2}}\int_{0}^{N}e^{-\lambda w}V(-w)U(w)dw.
\end{equation*}%
Since 
\begin{equation*}
\int_{0}^{\infty }e^{-\lambda w}V(-w)U(w)dw<\infty ,
\end{equation*}%
we conclude that, as $\min \left( r,n-r\right) \rightarrow \infty $%
\begin{align}
\mathbf{E}\left[ e^{-\lambda S_{r}}\mathbf{P}\left( L_{n-r}^{\prime }\geq
-S_{r}|S_{r}\right) ;L_{r}\geq 0\right]  
&\sim \frac{C_{1}}{r^{3/2}\left( n-r\right) ^{1/2}}\int_{0}^{\infty
}e^{-\lambda w}V(-w)U(w)dw  \notag \\
&=\frac{C_{1}}{r^{3/2}\left( n-r\right) ^{1/2}}\int_{0}^{\infty }V(-w)\mu
_{\lambda }(dw).  \label{Asympy_ExpS_j}
\end{align}%
This proves (\ref{D_1}). 

Equality (\ref{D_2}) derives from similar arguments. 

Relation (\ref{D_3})\ follows from the equality%
\begin{equation*}
\mathbf{E}[e^{\lambda S_{r}};\tau (n)=r]=\mathbf{E}[e^{\lambda S_{r}};\tau
(r)=r]\mathbf{P}\left( L_{n-r}\geq 0\right) 
\end{equation*}%
and estimates (\ref{AsymConditional}) and (\ref{Asym0}) with $x=0$.

This ends the proof of Lemma \ref{L_simpleconditional}
\end{proof}

For a fixed positive integer $N\leq \min (j/2,n-j)$ set 
\begin{equation*}
\mathcal{K}_{1}=\mathcal{K}_{1}(j,N):=\left\{ k\in \left[ N,j-N\right]
\right\} ,\ \mathcal{K}_{2}=\mathcal{K}_{2}(j,n,N):=\left\{ l\in \left[ j+N,n%
\right] \right\} .
\end{equation*}

The two next lemmas quantify the expectation of some exponential functionals of the random walk 
$\mathbf{S}$ when $n$ is large and $\tau(n)$ belongs to $\mathcal{K}_{1}$ or $\mathcal{K}_{2}$.
These results will be needed in the proof of \eqref{AsymInterm}.

\begin{lemma}
\label{L_Negl} If conditions (\ref{CondRW}) are valid then for every $%
\varepsilon >0$ there exists $N=N(\varepsilon )$ such that 
\begin{equation*}
\mathbf{E}\left[ e^{-S_{j}}e^{S_{\tau (j-1)}}e^{S_{\tau (n)}};\tau (n)\in 
\mathcal{K}_{2}\right] \leq \frac{\varepsilon }{j^{3/2}\sqrt{n-j}}
\end{equation*}%
for all $j\geq j_{0}=j_{0}\left( \varepsilon \right) ,n\geq
n_{0}=n_{0}\left( \varepsilon \right) $.
\end{lemma}

\begin{proof} Set for the sake of readability
\begin{align*}
R(j,n) :&=\mathbf{E}\left[ e^{-S_{j}}e^{S_{\tau (j-1)}}e^{S_{\tau
(n)}};\tau (j-1)<\tau (n)\right]  \\
&=\sum_{k=0}^{j-1}\sum_{l=j}^{n}\mathbf{E}\left[ e^{-S_{j}}e^{S_{\tau
(j-1)}}e^{S_{\tau (n)}};\tau (j-1)=k,\tau (n)=l\right] .
\end{align*}%
Observe that for $l\geq j$%
\begin{equation*}
\mathbf{E}\left[ e^{-S_{j}}e^{S_{\tau (j-1)}}e^{S_{\tau (n)}};\tau
(j-1)=k,\tau (n)=l\right] =\mathbf{E}\left[ e^{S_{k}}e^{S_{l}-S_{j}};\tau
(j-1)=k,\tau (l)=l\right] \mathbf{P}\left( L_{n-l}\geq 0\right) .
\end{equation*}%
Put
\begin{equation*}
S_{r}^{\prime }=S_{l}-S_{l-r},\quad r=0,1,...,l
\end{equation*}%
and denote $\mu ^{\prime }(m_{1},m_{2})$ the last point of maximum of $%
\left\{ S_{r}^{\prime },r=0,1,...,l\right\} $ on the interval $%
m_{1},m_{1}+1,...,m_{2}$. Using the independence of $S_{l}^{\prime
}-S_{l-k}^{\prime }$ and $S_{1}^{\prime },...,S_{l-k}^{\prime }$ we obtain, as $k<l$,
\begin{align*}
\mathbf{E}&\left[ e^{S_{l}}e^{S_{k}-S_{l}}e^{S_{l}-S_{j}};\tau (j-1)=k,\tau
(l)=l\right]  \\
&=\mathbf{E}\left[ e^{S_{l}^{\prime }}e^{-S_{l-k}^{\prime
}}e^{S_{l-j}^{\prime }};\mu ^{\prime }(l-j+1,l)=l-k,M_{l}^{\prime }<0\right] 
\\
&=\mathbf{E}\left[ e^{S_{l}^{\prime }-S_{l-k}^{\prime }}e^{S_{l-j}^{\prime
}};\max_{l-j+1\leq r\leq l-k}(S_{r}^{\prime }-S_{l-k}^{\prime })\leq
0,\max_{l-k< r\leq l}(S_{r}^{\prime }-S_{l-k}^{\prime })<0,M_{l}^{\prime
}<0\right]  \\
&=\mathbf{E}\left[ e^{S_{l-j}^{\prime }};\max_{l-j+1\leq r\leq
l-k}(S_{r}^{\prime }-S_{l-k}^{\prime })\leq 0,M_{l-k}^{\prime }<0\right] 
\mathbf{E}\left[ e^{S_{k}^{\prime }};M_{k}^{^{\prime }}<0\right] ,
\end{align*}%
where $M'$ is defined as $M$ in \eqref{def_LM} for the random walk $\mathbf{S}'$.
Thus,%
\begin{equation*}
R(j,n)=\sum_{k=0}^{j-1}\mathbf{E}\left[ e^{S_{k}^{\prime }},M_{k}^{\prime }<0%
\right] \sum_{l=j}^{n}\mathbf{E}\left[ e^{S_{l-j}^{\prime }};\max_{l-j+1\leq
r\leq l-k}(S_{r}^{\prime }-S_{l-k}^{\prime })\leq 0,M_{l-k}^{\prime }<0%
\right] \mathbf{P}\left( L_{n-l}\geq 0\right) .
\end{equation*}

For $x<0$ set 
\begin{equation*}
\Theta _{n}(x):=\mathbf{P}\left( \mu ^{\ast }(n)=n;S_{n}^{\ast }<-x\right)
\end{equation*}%
where the sequence $\mathbf{
S}^{\ast}:=\left\{ S_{m}^{\ast },m \in \N_0\right\} $ is an
independent copy of $\mathbf{
S}^{\prime}=\left\{ S_{m}^{\prime },m \in \N_0\right\} $ and $\mu
^{\ast }(j-k)$ is the last point of maximum of the random walk $\mathbf{
S}^{\ast}$ up to time $j-k$. Using duality we have

\begin{equation*}
\Theta _{n}(x)=\mathbf{P}\left( L_{n}^{\ast }\geq 0,S_{n}^{\ast }<-x\right) .
\end{equation*}%
Thus, 
\begin{equation*}
\mathbf{E}\left[ e^{S_{l-j}^{\prime }};\max_{l-j+1\leq r\leq
l-k}(S_{r}^{\prime }-S_{l-k}^{\prime })\leq 0,M_{l-k}^{\prime }<0\right] =%
\mathbf{E}\left[ e^{S_{l-j}^{\prime }}\Theta _{j-k-1}(S_{l-j}^{\prime
});M_{l-j}^{\prime }<0\right]
\end{equation*}%
and 
\begin{equation*}
R(j,n)=\sum_{k=0}^{j-1}\mathbf{E}\left[ e^{S_{k}};M_{k}<0\right]
\sum_{l=j}^{n}\mathbf{E}\left[ e^{S_{l-j}^{\prime }}\Theta
_{j-k-1}(S_{l-j}^{\prime });M_{l-j}^{\prime }<0\right] \mathbf{P}\left(
L_{n-l}\geq 0\right) .
\end{equation*}

From Proposition 2.3 in \cite{ABKV} as well as the monotonicity of the function $V$, we obtain that there exists a constant $C_{1}$ such
that for $n$ and $r\geq 1$%
\begin{equation}
\Theta _{n}(-r)=\mathbf{P}\left( L_{n}\geq 0,S_{n}<r\right) \leq \frac{C_{1}%
}{n^{3/2}}rV(-r).  \label{Asym11}
\end{equation}%
Thus, 
\begin{equation*}
\mathbf{E}\left[ e^{S_{l-j}^{\prime }}\Theta _{j-k-1}(S_{l-j}^{\prime
});M_{l-j}^{\prime }<0\right] \leq \frac{C_{1}}{\left( j-k\right) ^{3/2}}%
\mathbf{E}\left[ e^{S_{l-j}^{\prime }}\left\vert S_{l-j}^{\prime
}\right\vert V(S_{l-j}^{\prime });M_{l-j}^{\prime }<0\right] .
\end{equation*}%
Since $V$ is a renewal function (see Lemma 4.1 in \cite{GV2017} for instance), there is a constant $C$ such that $%
V(-y)\leq C\left( y+1\right) $ for all $y\geq 0$ (see \cite{F08} Ch. XI for instance). This, in turn, implies
existence of a constant $C_{2}$ such that%
\begin{equation*}
e^{-y/2}yV(-y)\leq Ce^{-y/2}\left( y+1\right) ^{2}\leq C_{1}
\end{equation*}%
for all $y\geq 0$. Combining this estimate with (\ref{AsymConditional}) we
see that there are constants $C_{1},C_{2}$ and $C_{3}$ such that%
\begin{equation*}
\mathbf{E}\left[ e^{S_{l-j}^{\prime }}\left\vert S_{l-j}^{\prime
}\right\vert V(S_{l-j}^{\prime });M_{l-j}^{\prime }<0\right] \leq C_{1}%
\mathbf{E}\left[ e^{S_{l-j}^{\prime }/2};M_{l-j}^{\prime }<0\right] \leq 
\frac{C_{2}}{\left( l-j\right) ^{3/2}}
\end{equation*}%
and thus
\begin{equation*}
\mathbf{E}\left[ e^{S_{l-j}^{\prime }}\Theta _{j-k}(S_{l-j}^{\prime
});M_{l-j}^{\prime }<0\right] \leq \frac{C_{3}}{\left( j-k\right) ^{3/2}}%
\frac{1}{\left( l-j\right) ^{3/2}}.
\end{equation*}%
Now we have, using the previous estimates,
\begin{align*}
R:&=\mathbf{E}\left[ e^{-S_{j}}e^{S_{\tau (j-1)}}e^{S_{\tau (n)}};\tau
(n)\in \mathcal{K}_{2}\right]  \\
&=\sum_{k=0}^{j-1}\mathbf{E}\left[ e^{-S_{k}},L_{k}\geq 0\right] \sum_{l\in 
\mathcal{K}_{2}}\mathbf{E}\left[ e^{S_{l-j}^{\prime }}\Theta
_{j-k}(S_{l-j}^{\prime });M_{l-j}^{\prime }<0\right] \mathbf{P}\left(
L_{n-l}\geq 0\right)  \\
&\leq C_{3}\sum_{k=0}^{j-1}\mathbf{E}\left[ e^{-S_{k}},L_{k}\geq 0\right] 
\frac{1}{\left( j-k\right) ^{3/2}}\sum_{l\in \mathcal{K}_{2}}\frac{1}{\left(
l-j\right) ^{3/2}}\mathbf{P}\left( L_{n-l}\geq 0\right)  \\
&\leq C_{4}\sum_{k=0}^{j-1}\frac{1}{\left( k+1\right) ^{3/2}}\frac{1}{%
\left( j-k+1\right) ^{3/2}}\sum_{l\in \mathcal{K}_{2}}\frac{1}{\left(
l-j\right) ^{3/2}}\frac{1}{\left( n-l+1\right) ^{1/2}},
\end{align*}%
where we applied \eqref{AsymConditional} and \eqref{Asym1}.
Since%
\begin{equation*}
\sum_{k=0}^{j-1}\frac{1}{\left( k+1\right) ^{3/2}}\frac{1}{\left( j-k\right)
^{3/2}}\leq \frac{C}{j^{3/2}}
\end{equation*}%
and 
\begin{eqnarray*}
\sum_{l\in \mathcal{K}_{2}}\frac{1}{\left( l-j\right) ^{3/2}}\frac{1}{\left(
n-l+1\right) ^{1/2}} &=&\sum_{k=N}^{n-j}\frac{1}{k^{3/2}}\frac{1}{\left(
n-j-k+1\right) ^{1/2}} \\
&=&\left( \sum_{k=N}^{\left( n-j\right) /2}+\sum_{k=(n-j)/2+1}^{n-j}\right) 
\frac{1}{k^{3/2}}\frac{1}{\left( n-j-k+1\right) ^{1/2}} \\
&\leq &\frac{C}{\left( n-j\right) ^{1/2}}\sum_{k=N}^{\left( n-j\right) /2}%
\frac{1}{k^{3/2}}+\frac{C}{\left( n-j\right) ^{3/2}}\sum_{k=1}^{\left(
n-j\right) /2}\frac{1}{k^{1/2}} \\
&\leq &\frac{\varepsilon }{2\left( n-j\right) ^{1/2}}+O\left( \frac{1}{n-j}%
\right) \leq \frac{\varepsilon }{\left( n-j\right) ^{1/2}}
\end{eqnarray*}%
for sufficiently large $N=N(\varepsilon )$, the desired estimate follows.

This end the proof of the Lemma.
\end{proof}

\begin{lemma}
\label{L_Negl2} If conditions (\ref{CondRW}) are valid then for every $%
\varepsilon >0$ there exists $N=N(\varepsilon )$ such that 
\begin{equation*}
\mathbf{E}\left[ e^{-S_{j}}e^{S_{\tau (j-1)}}e^{S_{\tau (n)}};\tau (n)\in 
\mathcal{K}_{1}\right] \leq \frac{\varepsilon }{j^{3/2}\sqrt{n-j}}
\end{equation*}%
for all $j\geq j_{0}=j_{0}\left( \varepsilon \right) ,n\geq
n_{0}=n_{0}\left( \varepsilon \right) $.
\end{lemma}

\begin{proof} Applying \eqref{AsymConditional} and \eqref{Asym1} we conclude
as before that for $m<n$ 
\begin{eqnarray*}
\mathbf{E}\left[ e^{-S_{m}};L_{n}\geq 0\right] &=&\mathbf{E}\left[ e^{-S_{m}}%
\mathbf{P}\left( L_{n-m}^{^{\prime }}\geq -S_{m}|S_{m}\right) ;L_{m}\geq 0%
\right] \\
&\leq &\frac{C}{\left( n-m\right) ^{1/2}}\mathbf{E}\left[
e^{-S_{m}}V(-S_{m});L_{m}\geq 0\right] \\
&\leq &\frac{C_{1}}{\left( n-m\right) ^{1/2}}\mathbf{E}\left[
e^{-S_{m}/2};L_{m}\geq 0\right] \leq \frac{C_{2}}{\left( n-m\right) ^{1/2}}%
\frac{1}{m^{3/2}}.
\end{eqnarray*}%
Therefore,%
\begin{eqnarray*}
\mathbf{E}\left[ e^{-S_{j}}e^{S_{\tau (j-1)}}e^{S_{\tau (n)}};\tau (n)\in 
\mathcal{K}_{1}\right] &=&\sum_{k=N}^{j-N}\mathbf{E}\left[
e^{S_{k}-S_{j}}e^{S_{k}};\tau (n)=k\right] \\
&=&\sum_{k=N}^{j-N}\mathbf{E}\left[ e^{S_{k}};\tau (k)=k\right] \mathbf{E}%
\left[ e^{-S_{j-k}};L_{n-k}\geq 0\right] \\
&\leq &\frac{C_{3}}{\left( n-j\right) ^{1/2}}\sum_{k=N}^{j-N}\frac{1}{k^{3/2}%
}\frac{1}{\left( j-k\right) ^{3/2}} \\
&\leq &\frac{C_{4}}{\left( n-j\right) ^{1/2}j^{3/2}}\sum_{k=N}^{\infty }%
\frac{1}{k^{3/2}}\leq \frac{\varepsilon }{\left( n-j\right) ^{1/2}j^{3/2}}
\end{eqnarray*}%
for sufficiently large $N=N(\varepsilon )$.
\end{proof}

\subsection{Proof of relation (\protect\ref{AsymInterm}).}

In order to prove (\ref{AsymInterm}) we use a more convenient representation
for $\mathbf{P}\left( \mathcal{A}_{i}(n)\right) $. 
Recalling \eqref{expr_PAin} and Corollary \ref{C_fractional}, we get

\begin{eqnarray*}
\mathbf{P}\left( \mathcal{A}_{i}(n)\right)  &=&\mathbf{E}\left[ \frac{a_{i}}{%
a_{n}+b_{n}-b_{i+1}}\frac{a_{n}}{a_{n}+b_{n}-b_{1}}\right]  
=\mathbf{E}\left[ \frac{e^{-S_{i}}}{\sum_{k=i+1}^{n}e^{-S_{k}}}\frac{%
e^{-S_{n}}}{\sum_{k=1}^{n}e^{-S_{k}}}\right]\\ &=&\mathbf{E}\left[ \frac{%
e^{S_{n}-S_{i}}}{\sum_{k=i+1}^{n}e^{S_{n}-S_{k}}}\frac{1}{%
\sum_{k=1}^{n}e^{S_{n}-S_{k}}}\right]  
=\mathbf{E}\left[ \frac{%
e^{S_{n-i}}}{\sum_{k=0}^{n-i-1}e^{S_{k}}}\frac{1}{%
\sum_{k=0}^{n-1}e^{S_{k}}}\right]  ,
\end{eqnarray*}
where we have used that $(X_1,...,X_n)$ is distributed as $(X_n,...,X_1)$.
Hence, if we introduce a new BPRE with i.i.d. probability generating functions 
\begin{equation*}
\bar{F}_{k}(s):=\frac{1}{1+m(\bar{F}_{k})(1-s)}=\frac{1}{1+e^{-X_{k}}(1-s)},\quad k \in \N,
\end{equation*}%
and $\bar{X}_k:=\log \bar{F}_k(1)=-X_k$ for generation $k$, we obtain from \eqref{Frac_F} that
\begin{equation} \mathbf{P}\left( \mathcal{A}_{i}(n)\right) = \mathbf{E}\left[ e^{-\bar{S}_{n-i}}\left( 1-\bar{F}_{0,n-i-1}(0)\right)
 \left( 1-\bar{F}_{0,n-1}(0)\right) \right] .\label{Exp_j}\end{equation}
For the sake of readability, we will write $j$ instead of $n-i$ in the remaining part of the proofs.
Let us introduce: 
\begin{equation*}
\mathcal{\bar{H}}_{j,n}:=e^{-\bar{S}_{j}}\left( 1-\bar{F}_{0,j-1}(0)\right)
\left( 1-\bar{F}_{0,n-1}(0)\right) ,
\end{equation*}%
put $\bar{S}_{k}:=-S_{k},k \in \N_0$ and let $\bar{\tau}(n):=\min \left\{
k\geq 0:\bar{S}_{k}=\bar{L}_{n}\right\} $ with $\bar{L}_{n}:=\min_{0\leq
r\leq n}\bar{S}_{r}$.

We first show that for large but fixed $N$ the quantities 
\begin{equation*}
\mathbf{E}\left[ \mathcal{\bar{H}}_{j,n};\bar{\tau}(n)\in \left[ N,j-N\right]
\right] 
\quad
\text{and}
\quad
\mathbf{E}\left[ \mathcal{\bar{H}}_{j,n};\bar{\tau}(n)\in \left[ j+N,n\right]
\right] 
\end{equation*}%
give, as $\min \left( j,n-j\right) \rightarrow \infty $ a negligible
contribution to $\mathbf{E}\left[ \mathcal{\bar{H}}_{j,n}\right] $ in
comparison with $j^{-3/2}\left( n-j\right) ^{-1/2}$ and then demonstrate
that $\mathbf{E}\left[ \mathcal{\bar{H}}_{j,n}\right] \sim Cj^{-3/2}\left(
n-j\right) ^{-1/2}$ for $C>0.$

\begin{lemma}
\label{L_neglLeftInterm}For any $\varepsilon >0$ there exists $%
N\left( \varepsilon \right) $ such that, for all $N\geq N(\varepsilon)$ and
all sufficiently large $j$ and $n-j$ 
\begin{equation*}
\mathbf{E}\left[ \mathcal{\bar{H}}_{j,n} ;\bar{\tau}(n)\in \mathcal{K}_{1}\cup \bar{\tau}%
(n)\in \mathcal{K}_{2}\right] \leq \frac{\varepsilon }{j^{3/2}\left(
n-j\right) ^{1/2}}.
\end{equation*}
\end{lemma}

\begin{proof} We know that $1-\bar{F}_{0,m}(0)\leq e^{\min_{0\leq k\leq m}%
\bar{S}_{k}}=e^{\bar{S}_{\bar{\tau}(m)}}$. Thus, 
\begin{equation*}
e^{-\bar{S}_{j}}\left( 1-\bar{F}_{0,j-1}(0)\right) \left( 1-\bar{F}%
_{0,n-1}(0)\right) \leq e^{-\bar{S}_{j}}e^{\bar{S}_{\bar{\tau}(j-1)}}e^{\bar{%
S}_{\bar{\tau}(n-1)}}
\end{equation*}%
Lemma \ref{L_neglLeftInterm} thus follows from Lemmas \ref{L_Negl} and \ref%
{L_Negl2}.\end{proof}

We now have all the ingredients to conclude the proof of Theorem \ref{T_total}.\\

\begin{proof}[Proof of point 3) of Theorem \ref{T_total}]
We have the decomposition 
\begin{align*}
\mathbf{E}\left[ \mathcal{\bar{H}}_{j,n}\right]  =&\mathbf{E}\left[ 
\mathcal{\bar{H}}_{j,n};\bar{\tau}(n)\in \mathcal{K}_{1}\cup \bar{\tau}%
(n)\in \mathcal{K}_{2}\right] 
+\mathbf{E}\left[ \mathcal{\bar{H}}_{j,n};\bar{\tau}(n)<N\right]\\ &+\mathbf{E%
}\left[\mathcal{\bar{H}}_{j,n};\bar{\tau}(n)\in (j-N,j)\right]
+\mathbf{E}\left[ \mathcal{\bar{H}}_{j,n};\bar{\tau}(n)\in \lbrack j,j+N)%
\right] .
\end{align*}%
We know from the previous lemma that, for any $\varepsilon >0$ there
exists $N=N(\varepsilon )$ such that, given $\min \left( j,n-j\right) $ is
large, the first term in the right-hand side satisfies:

\begin{equation}
\mathbf{E}\left[ \mathcal{\bar{H}}_{j,n};\bar{\tau}(n)\in \mathcal{K}%
_{1}\cup \bar{\tau}(n)\in \mathcal{K}_{2}\right] \leq \frac{\varepsilon }{%
j^{3/2}\left( n-j\right) ^{1/2}}.  \label{Term0}
\end{equation}%
The remaining proof is splitted into several steps.

1) Let $k<N<j$. By conditioning on the trajectory until time $k$ we obtain
\begin{equation*}
\mathbf{E}\left[ \mathcal{\bar{H}}_{j,n};\bar{\tau}(n)=k\right] =\mathbf{E}%
\left[ e^{-\bar{S}_{k}}\Lambda _{j-k,n-k}\left( Z_{k1},Z_{k2}\right) ;\bar{%
\tau}(k)=k\right]
\end{equation*}%
where $Z_{k1},Z_{k2}$ are independent random variables conditionally on $(S_1,...,S_k)$, and distributed as $Z_{k}$, and 
\begin{equation*}
\Lambda _{j,n}\left( z_{1},z_{2}\right) =\mathbf{E}\left[ e^{-\bar{S}%
_{j}}\left( 1-\bar{F}_{0,j-1}^{z_{1}}(0)\right) \left( 1-\bar{F}%
_{0,n-1}^{z_{2}}(0)\right) ;\bar{L}_{n}\geq 0\right] .
\end{equation*}%
Set $t=\left[ j/2\right] $ and denote 
\begin{align*}
G_{j} &:=\sum_{r=0}^{t}e^{-\bar{S}_{r}},\ H_{j}:=\sum_{r=0}^{j-t}e^{\bar{S}_{r}},\ 
\tilde{H}_{j}:=\sum_{r=t+1}^{j-1}e^{\bar{S}_{j}-\bar{S}_{r}}, \\
T_{n-j} &:=\sum_{r=0}^{n-j}e^{-\bar{S}_{r}},\ \hat{T}_{n-j}:=%
\sum_{r=j}^{n-1}e^{\bar{S}_{j}-\bar{S}_{r}}.
\end{align*}
As $\min (j,n-j)\rightarrow \infty $,
\begin{eqnarray*}
G_{j} &\rightarrow &G_{\infty }:=\sum_{r=0}^{\infty }e^{-S_{r}}\quad \mathbf{%
P}^{+}\text{-a.s.,} \\
H_{j} &\rightarrow &H_{\infty }:=\sum_{r=0}^{\infty }e^{S_{r}}\;\mathbf{P}_x
^{-}\text{-a.s.,} \quad \forall x \leq 0 \\
T_{n-j} &\rightarrow &T_{\infty }:=\sum_{r=0}^{\infty }e^{-S_{r}}\quad 
\mathbf{P}_x^{+}\text{-a.s.}, \quad \forall x \geq 0.
\end{eqnarray*}%
In addition
\begin{equation*}
\Lambda _{j,n}\left( z_{1},z_{2}\right) =\mathbf{E}\left[ e^{-\bar{S}%
_{j}}\varphi (G_{j},\tilde{H}_{j},S_{j};\hat{T}_{n-j});\bar{L}_{n}\geq 0%
\right] ,
\end{equation*}%
where (recall (\ref{Exp_j}))%
\begin{equation*}
\varphi (u,v,z;t):=\left( 1-\left(1- \frac{1}{u+e^{-z}v}\right)
^{z_{1}}\right) \left( 1-\left(1- \frac{1}{u+e^{-z}(v+t)}\right)
^{z_{2}}\right)
\end{equation*}
is bounded as $u \geq 1$ and $v,z_1,z_2 \geq 0$ by definition.
We may thus apply Lemma \ref{L_pr4New} and obtain that 
\begin{equation*}
\lim_{\min (j,n-j)\rightarrow \infty }\frac{\Lambda _{j,n}\left(
z_{1},z_{2}\right) }{\mathbf{E}[e^{-\bar{S}_{j}};\bar{L}_{n}\geq 0]\ }%
:=\Lambda _{\infty }\left( z_{1},z_{2}\right)
\end{equation*}%
exists for each fixed pair $(z_{1},z_{2})$. Hence, invoking the dominated
convergence theorem
we conclude that for any $k<N$, as $%
\min \left( j,n-j\right) \rightarrow \infty $%
\begin{multline*}
\mathbf{E}\left[ e^{-\bar{S}_{k}}\Lambda _{j-k,n-k}\left(
Z_{k1},Z_{k2}\right) ;\bar{\tau}(k)=k\right]  \\
\sim \mathbf{E}\left[ e^{-\bar{S}_{k}}\Lambda _{\infty }\left(
Z_{k1},Z_{k2}\right) ;\bar{\tau}(k)=k\right] \mathbf{E}[e^{-\bar{S}_{j}};%
\bar{L}_{n}\geq 0] \\ = C(k)  \mathbf{E}[e^{-\bar{S}_{j}};%
\bar{L}_{n}\geq 0].
\end{multline*}
Applying \eqref{D_1} and summing over the $k$'s in $[0,N-1]$, give, as $\min (j,n-j) \to \infty$
\begin{eqnarray}
\mathbf{E}\left[ \mathcal{\bar{H}}_{j,n};\bar{\tau}(n)<N\right]&\sim &\frac{C_{1}}{j^{3/2}\left( n-j\right) ^{1/2}}\int_{0}^{\infty
}V(-w)\mu _{1}(dw).  \label{Term1}
\end{eqnarray}

2) Now we take $1\leq k<N$, set $t=\left[ j/2\right] $ and denote 
\begin{align*}
G_{j} &:=\sum_{r=0}^{t}e^{-\bar{S}_{r}},\ H_{j,k}:=\sum_{r=0}^{j-t-k}e^{\bar{%
S}_{r}},\ \tilde{H}_{j,k}:=\sum_{r=t+1}^{j-k-1}e^{\bar{S}_{j-k}-\bar{S}_{r}},
\\
T_{0,k} &:=\sum_{r=0}^{k-1}e^{-\bar{S}_{r}},\ \hat{T}_{j-k,j}:=%
\sum_{r=j-k}^{j-1}e^{\bar{S}_{j-k}-\bar{S}_{r}}, \\
T_{0,n-j+k} &:=\sum_{r=0}^{n-j+k-1}e^{-\bar{S}_{r}},\ \hat{T}%
_{j-k,n}:=\sum_{r=j-k}^{n-1}e^{\bar{S}_{j-k}-\bar{S}_{r}}.
\end{align*}%
As $\min (j,n-j)\rightarrow \infty $
$$
G_{j} \rightarrow G_{\infty }\quad \mathbf{P}_x^{+}\text{-a.s.,} \quad \forall x \geq 0 \quad 
H_{j,k} \rightarrow H_{\infty } \quad \mathbf{P}^{-}\text{-a.s.}\quad \text{and} \quad
T_{0,n-j+k} \rightarrow T_{\infty }\quad \mathbf{P}^{+}\text{-a.s.}
$$
Besides,%
\begin{align*}
\mathbf{E}&\left[ \mathcal{\bar{H}}_{j,n};\bar{\tau}(n)=j-k\right] \\
&=\mathbf{E}\left[ e^{\bar{S}_{j}}\frac{1}{\sum_{r=0}^{j-1}e^{\bar{S}_{j}-%
\bar{S}_{r}}}\frac{1}{\sum_{r=0}^{j-1}e^{\bar{S}_{j}-\bar{S}%
_{r}}+\sum_{r=j}^{n-1}e^{-(\bar{S}_{r}-\bar{S}_{j})}};\bar{\tau}(n)=j-k%
\right] \\
&=\mathbf{E}\left[ e^{\bar{S}_{j-k}}\frac{e^{\bar{S}_{j-k}-\bar{S}_{j}}}{e^{\bar{S}%
_{j-k}}G_{j}+\tilde{H}_{j,k}+\hat{T}_{j-k,j}}\frac{1}{e^{\bar{S}_{j-k}}G_{j}+%
\tilde{H}_{j,k}+\hat{T}_{j-k,n}};\bar{\tau}(n)=j-k\right] \\
&=\mathbf{E}\left[ e^{\bar{S}_{j-k}}\varphi (G_{j},\tilde{H}_{j,k},\bar{S}%
_{j-k};\hat{T}_{j-k,j},\hat{T}_{j-k,n},\bar{S}_{j-k}-\bar{S}_{j});\bar{\tau}(n)=j-k%
\right] .
\end{align*}%
Recalling Lemma \ref{L_pr6New} we see that, for each fixed $k$ there exists a
constant $J_{-k}\geq 0$ such that, as $\min (j,n-j)\rightarrow \infty $ 
\begin{equation}
\mathbf{E}\left[ \mathcal{\bar{H}}_{j,n};\bar{\tau}(n)=j-k\right] \sim J_{-k}%
\mathbf{E}[e^{\bar{S}_{j-k}};\tau (n)=j-k]\sim \frac{CJ_{-k}}{j^{3/2}\left(
n-j\right) ^{1/2}}.  \label{Term2}
\end{equation}

3) Now we take $0\leq k<N$, set $t=\left[ j/2\right] ,$ denote 
\begin{align*}
G_{j} &:=\sum_{r=0}^{t}e^{-\bar{S}_{r}},\quad H_{0,j-t}:=\sum_{r=0}^{j-t-2}e^{\bar{%
S}_{r}}, \quad  \tilde{H}_{t+1,j-1}:=\sum_{r=t+1}^{j-1}e^{\bar{S}_{j+k}-\bar{S}_{r}},
\\
H_{0,j+k-t} &:=\sum_{r=0}^{j+k-t-2}e^{\bar{S}_{r}},\quad \tilde{H}%
_{t+1,j+k-1}:=\sum_{r=t+1}^{j+k-1}e^{\bar{S}_{j+k}-\bar{S}_{r}}, \quad
\Delta _{k} :=\bar{S}_{k},\\ \tilde{\Delta}_{j,k}&:=\bar{S}_{j+k}-\bar{S}_{j}
\quad
T_{0,n-j+k} :=\sum_{r=0}^{n-j-k-1}e^{-\bar{S}_{r}},\quad \hat{T}%
_{j+k,n}:=\sum_{r=j+k}^{n-1}e^{\bar{S}_{j+k}-\bar{S}_{r}}
\end{align*}%
and evaluate%
\begin{eqnarray*}
&&\mathbf{E}\left[ \mathcal{\bar{H}}_{j,n};\bar{\tau}(n)=j+k\right]  \\
&=&\mathbf{E}\left[ e^{\bar{S}_{j+k}}\frac{1}{\sum_{r=0}^{j-1}e^{\bar{S}%
_{j+k}-\bar{S}_{r}}}\frac{e^{\bar{S}_{j+k}-\bar{S}_{j}}}{\sum_{r=0}^{j-1}e^{%
\bar{S}_{j+k}-\bar{S}_{r}}+\sum_{r=j}^{n-1}e^{-(\bar{S}_{r}-\bar{S}_{j+k})}};%
\bar{\tau}(n)=j+k\right]  \\
&=&\mathbf{E}\left[ e^{\bar{S}_{j+k}}\frac{1}{e^{\bar{S}_{j+k}}G_{j}+\tilde{H%
}_{t+1,j-1}}\frac{e^{\bar{S}_{j+k}-\bar{S}_{j}}}{e^{\bar{S}_{j+k}}G_{j}+%
\tilde{H}_{t+1,j+k-1}+\hat{T}_{j+k,n}};\bar{\tau}(n)=j+k\right]  \\
&=&\mathbf{E}\left[ e^{\bar{S}_{j-k}}\varphi (G_{j},\tilde{H}_{t+1,j-1},\bar{%
S}_{j+k};\tilde{H}_{t+1,j+k-1},\hat{T}_{j+k,n},\bar{S}_{j+k}-\bar{S}_{j});%
\bar{\tau}(n)=j+k\right] 
\end{eqnarray*}%
with evident meaning for $\varphi $. As $\min (j,n-j)\rightarrow
\infty $%
\begin{eqnarray*}
G_{j} &\rightarrow &G_{\infty }\quad \mathbf{P}_x^{+}\text{-a.s.,} \quad \forall x \geq 0, \\
\left( H_{0,j-t},H_{0,j+k-t},\Delta _{k}\right)  &\rightarrow &(H_{\infty
},H_{\infty },\Delta _{k})\;\;\mathbf{P}^{-}\text{-a.s.,}  \\
T_{0,n-j+k} &\rightarrow &T_{\infty }\quad \mathbf{P}^{+}\text{-a.s.}
\end{eqnarray*}%
Hence, using Lemma \ref{L_pr6New} and Equation \eqref{D_1} we get that, for each fixed $k$ there
exists a constant $J_{+k}\geq 0$ such that, as $\min (j,n-j)\rightarrow
\infty $%
\begin{equation}
\mathbf{E}\left[ \mathcal{\bar{H}}_{j,n};\bar{\tau}(n)=j+k\right] \sim J_{+k}%
\mathbf{E}[e^{\bar{S}_{j+k}};\tau (n)=j+k]\sim \frac{CJ_{+k}}{j^{3/2}\left(
n-j\right) ^{1/2}}.  \label{Term3}
\end{equation}

Combining (\ref{Term0})-(\ref{Term3}) shows that 
\begin{equation*}
\lim_{\min (j,n-j)\rightarrow \infty }j^{3/2}\left( n-j\right) ^{1/2}\mathbf{%
E}\left[ \mathcal{\bar{H}}_{j,n}\right] =K\in \left( 0,\infty \right) .
\end{equation*}%
Recalling the substitution $i\rightarrow n-j$ we conclude that

\begin{equation*}
\lim_{\min (i,n-i)\rightarrow \infty }i^{1/2}\left( n-i\right) ^{3/2}\mathbf{%
P}\left( \mathcal{A}_{i}(n)\right) =K\in (0,\infty ).
\end{equation*}%
This completes the proof of Theorem \ref{T_total}.
\end{proof}

\begin{acknow}
 This work is supported by the Russian Science Foundation under the grant
19-11-00111.
\end{acknow}


\begin{thebibliography}{99}

\bibitem{AFa2013} V.I. Afanasyev, Conditional limit theorem for maximum of
random walk in a random environment, {\em Theory Probab. Appl.} 58(4),
525-545, 2014.

\bibitem{AFa2015} V.I. Afanasyev, On the Time of Reaching a High
Level by a Transient Random Walk in a Random Environment, {\em Theory Probab.
Appl.}, 61(2), 178--207, 2017

\bibitem{AFa2018} V.I. Afanasyev, On the non-recurrent random walk in
a random environment, {\em Discrete Math. Appl.},  28(3), 139--156, 2018.

\bibitem{AFa2019} V. I. Afanasyev, Two-Boundary Problem for a Random
Walk in a Random Environment. {\em Theory Probab. Appl.}
63(3), 339-350, 2019.

\bibitem{4h} V.I. Afanasyev, J. Geiger, G. Kersting, V.A. Vatutin,
Criticality for branching processes in random environment, {\em Ann. Probab.}
33(2), 645-673, 2005.

\bibitem{ABKV} V.I. Afanasyev, C. Boeinghoff, G. Kersting, V.A. Vatutin,
Limit theorems for weakly subcritical branching processes in random
environment, {\em J. Theoret. Probab.} 25(3), 703-732, 2012.

\bibitem{A15} M. Arca, F. Mougel, T. Guillemaud, S. Dupas, Q. Rome, A. Perrard, F. Muller, A. Fossoud, C. Capdevielle-Dulac, 
M. Torres-Leguizamon, 
Reconstructing the invasion and the demographic history of the yellow-legged hornet, Vespa velutina, in Europe, 
{\em Biol. Invasions}, 17(8), 2357--2371, 2015.

\bibitem{Don12} R.A. Doney, Local behavior of first passage probabilities.
{\em Probab. Theory Relat. Fields}, 152, 559--588, 2012.

\bibitem{DLVZ19} E. Dyakonova, D. Li, V.A. Vatutin, M. Zhang, Branching
processes in random environment with immigration stopped at zero, {\em arXiv:
1905.03535}, 2019.

\bibitem{F08} W. Feller, An introduction to probability theory and its applications, vol II {\em John Wiley \& Sons}, 2008.

\bibitem{guivarc2001proprietes} Y. Guivarc'h, Q. Liu, Propri{\'e}t{\'e}s
asymptotiques des processus de branchement en environnement al{\'e}atoire,
{\em C. R. Acad. Sci. Paris S\'er. I Math.}, 332(4), 339--344, 2001.

\bibitem{HI96} P. Haccou , Y. Iwasa, Establishment probability in fluctuating environments: a branching process model.
{\em Theor. Pop. Biol.}, 50, 254--280, 1996

\bibitem{HV03} P. Haccou, V.A. Vatutin, Establishment success and extinction risk in autocorrelated environments.
{\em Theor. Pop. Biol.}, 64(3), 303--314, 2003

\bibitem{H92} I. Hanski, Inferences from ecological incidence functions.
{\em Am. Nat.}, 139(3), 657--662, 1992

\bibitem{GV2017} G. Kersting, V. Vatutin, Discrete time branching processes
in random environment, {\em ISTE \& Wiley}, 2017.

\bibitem{KKS75} H. Kesten, M.V. Kozlov, F. Spitzer, A limit law for random
walk in a random environment. {\em Compositio Math.}, 30, 145--168, 1975

\bibitem{K73} M.V. Kozlov, A random walk on the line with stochastic structure. {\em Theory Probab. Appl.}, 18, 406--408, 1973.

\bibitem{LVZ19} D. Li, V.A. Vatutin, M. Zhang,  Subcritical branching processes in random environment with immigration stopped at zero, {\em arXiv:
1906.09590}, 2019.

\bibitem{vatutin2013evolution} V.~A. Vatutin, E.~E. Dyakonova, S.~Sagitov. Evolution of branching processes in a random
environment. {\em Proc. Steklov Inst. Math.},  282(1), 220--242, 2013.


\bibitem{V06} C. Villemant, J-C. Streito, J. Haxaire, Premier bilan de l'invasion de Vespa velutina Lepeletier en France (Hymenoptera, Vespidae),
{\em Bull. Soc. Entom. France},  111(4), 535--538, 2006.

\bibitem{VW07} V.A. Vatutin, V. Wachtel, Local probabilities for random
walks conditioned to stay positive. {\em Probab. Theory Relat. Fields}, 
143(1-2), 177--217, 2009.


\end{thebibliography}
\end{document}